\documentclass[11pt]{amsart}
\usepackage{amssymb}
\usepackage{amsmath}
\usepackage{amsfonts}
\usepackage{graphicx}
\usepackage{hyperref}

\usepackage[total={17cm,22cm},top=2.5cm, left=2.3cm]{geometry}

\usepackage{hyperref}
    \usepackage{aeguill}
    \usepackage{type1cm}

\theoremstyle{plain}
\newtheorem{thm}{Theorem}[section]

\newtheorem{corollary}[thm]{Corollary}

\newtheorem{definition}[thm]{Definition}

\newtheorem{lemma}[thm]{Lemma}

\newtheorem{proposition}[thm]{Proposition}
\newtheorem{remark}[thm]{Remark}

\newtheorem{theorem}[thm]{Theorem}

\numberwithin{equation}{section}
\newcommand{\N}{\mathbb{N}}

\newcommand{\R}{\mathbb{R}}

\newcommand{\Rn}{\mathbb{R}^n}

\DeclareMathOperator{\lip}{Lip\,\!}

\usepackage[usenames,dvipsnames]{color}

\usepackage[dvipsnames]{xcolor}

\begin{document}

\title[$C^{1,\omega}$ extension formulas for $1$-jets on Hilbert spaces]{$C^{1,\omega}$ extension formulas for $1$-jets on Hilbert spaces}

\author{Daniel Azagra}
\address{ICMAT (CSIC-UAM-UC3-UCM), Departamento de An{\'a}lisis Matem{\'a}tico y Matem\'atica Aplicada,
Facultad Ciencias Matem{\'a}ticas, Universidad Complutense, 28040, Madrid, Spain. }
\email{azagra@mat.ucm.es}

\author{Carlos Mudarra}
\address{Aalto University, Department of Mathematics and Systems Analysis, P.O. BOX 11100, FI-00076 Aalto, Finland}
\email{carlos.mudarra@aalto.fi}

\date{December 12, 2019}



\thanks{D. Azagra and C. Mudarra were partially supported by Grant PGC2018-097286-B-I00. C. Mudarra also acknowledges financial support from the Academy of Finland}

\begin{abstract}
We provide necessary and sufficient conditions for a $1$-jet $(f, G):E\to \R\times X$ to admit an extension $(F, \nabla F)$ for some $F\in C^{1, \omega}(X)$. Here $E$ stands for an arbitrary subset of a Hilbert space $X$ and $\omega$ is a modulus of continuity. As a corollary, in the particular case $X=\R^n$, we obtain an extension (nonlinear) operator whose norm does not depend on the dimension $n$. Furthermore, we construct extensions $(F, \nabla F)$ in such a way that: (1) the (nonlinear) operator $(f, G)\mapsto (F, \nabla F)$ is bounded with respect to a natural seminorm arising from the constants in the given condition for extension (and the bounds we obtain are almost sharp); (2) $F$ is given by an explicit formula; (3) $(F, \nabla F)$ depend continuously on the given data $(f, G)$; (4) if $f$ is bounded (resp. if $G$ is bounded) then so is $F$ (resp. $F$ is Lipschitz). We also provide similar results on superreflexive Banach spaces. 
\end{abstract}

\maketitle

\section{Introduction and main results}

Throughout this paper we will assume that $\omega: [0,+\infty) \to [0,+ \infty)$ is a concave and increasing function such that $\omega(0)=0$ and $\lim_{t \to +\infty} \omega(t)=+\infty.$ Also, we will denote
\begin{equation}\label{definition of varphi}
\varphi(t)= \int_0^t \omega(s) ds
\end{equation}
for every $t \geq 0$, and if $X$ is a Banach space then $C^{1,\omega}(X)$ will stand for the set of all functions $g: X\to\R$ which are Fr\'echet differentiable and such that $Dg:X\to X^{*}$ is uniformly continuous, with modulus of continuity $\omega$, that is to say, there exists some constant $C>0$ such that
$$
\|D g(x)-D g(y)\|_*\leq C\omega(\|x-y\|)
$$
for all $x, y\in X$. Here $\|\cdot\|_{*}$ denotes the usual norm of the dual space $X^{*}$, defined by $\|\xi\|_{*}=\sup\{ \xi(x) : x\in X, \,\|x\|\leq 1\}$ for every $\xi\in X^{*}$.
 
If $E$ is a subset of $\R^n$ and we are given functions $f:E\to\R$, $G:E\to\R^n$, Glaeser's $C^{1,\omega}$ version of the classical Whitney extension theorem (see \cite{Whitney, Glaeser}) tells us that there exists a function $F\in C^{1,\omega}(\R^n)$ with $(F, \nabla F)=(f, G)$ on $E$ if and only if the $1$-jet $(f,G)$ satisfies the following property: there exists a constant $M>0$ such that
\begin{equation}\label{Whitneys condition}
|f(x)-f(y)-\langle G(y), x-y\rangle|\leq M\varphi(|x-y|), \,\,\, \textrm{ and } \,\,\,
|G(x)-G(y)|\leq M \omega(|x-y|)
\end{equation}
for all $x, y\in E$. We can trivially extend $(f,G)$ to the closure $\overline{E}$ of $E$ so that the inequalities \eqref{Whitneys condition} hold on $\overline{E}$ with the same constant $M$, and the function $F$ can be explicitly defined by
\begin{equation}\label{Whitney Extension Operator} 
F(x)=\begin{cases}
f(x) & \mbox{ if } x\in \overline{E} \\
\sum_{Q\in\mathcal{Q}}\left( f(x_Q)+\langle G(x_Q), x-x_Q\rangle\right)\psi_{Q}(x)  & \mbox{ if }  x\in\R^n\setminus \overline{E},
\end{cases}
\end{equation}
where $\mathcal{Q}$ is a family of {\em Whitney cubes} that cover the complement of $\overline{E}$, $\{\psi_{Q}\}_{Q\in\mathcal{Q}}$ is the usual Whitney partition of unity associated with $\mathcal{Q}$, and $x_Q$ is a point of $\overline{E}$ which minimizes the distance of $\overline{E}$ to the cube $Q$. Recall also that the function $F$ constructed in this way has the property that 
$$
M_{w}(\nabla F):=\sup_{x, y\in\R^n, x\neq y}\frac{|\nabla F(x)-\nabla F(y)|}{\omega\left( |x-y|\right)} \leq k(n)M,
$$ where $k(n)$ is a constant depending only on $n$ (but going to infinity as $n\to\infty$), and that the operator $(f, G)\to (F, \nabla F)$ thus obtained is linear.

The existence of some $M>0$ satisfying \eqref{Whitneys condition} is equivalent to saying that
\begin{equation}\label{Glaesers condition}
A(f,G):=\sup_{ x\in\R^n; \, y, z\in E, \, |x-y| + |x-z|> 0}\frac{ |f(y)+\langle G(y), x-y\rangle- f(z)-\langle G(z), x-z\rangle | }{ \varphi(|x-y|) +\varphi(|x-z|) } <\infty,
\end{equation}
and Whitney's theorem for $C^{1, \omega}(\R^n)$ can be restated as follows.
\begin{theorem}[Glaeser's version of Whitney's extension theorem for $C^{1, \omega}$; see \cite{Glaeser}]\label{WhitneyGlaeserThm}
For every $n\in\N$ there exists some constant $k(n)>0$, depending only on $n$, such that, for every $1$-jet $(f, G)$ defined on a subset $E$ of $\R^n$, we have that $A(f, G)<\infty$ if and only if there exists $F\in C^{1, \omega}(\R^n)$ such that $(F, \nabla F)=(f, G)$ on $E$ and $A(F, \nabla F)\leq k(n) A(f, G)$.
\end{theorem} 
Here we have denoted
\begin{equation}
A(F, \nabla F):=\sup_{ x, y, z\in\R^n, \, |x-y| + |x-z|> 0}\frac{ |F(y)+\langle \nabla F(y), x-y\rangle- F(z)-\langle \nabla F(z), x-z\rangle | }{ \varphi(|x-y|) +\varphi(|x-z|) }
\end{equation}
because $(F, \nabla F)$ can be regarded as a $1$-jet on $\R^n$.
More generally, if $X$ is a Banach space, $E$ is a subset of $X$, and $(f, G):E\to\R\times  X^{*}$, we will denote
\begin{equation}\label{Our optimal condition}
A(f,G; E):=\sup_{ x\in X; \, y, z\in E, \, \|x-y \| + \|x-z \|> 0}\frac{ |f(y)+\langle G(y), x-y\rangle- f(z)-\langle G(z), x-z\rangle | }{ \varphi(\|x-y\|) +\varphi(\|x-z\|) } <\infty,
\end{equation}
which will be shortened to $A(f, G)$ whenever the subset $E$ is understood. In particular, for a differentiable function $F:X\to\R$, we let $A(F, \nabla F)$ stand for $A(F, \nabla F; X)$.

As we said, if we construct such an $F$ by means of the Whitney Extension Operator \eqref{Whitney Extension Operator}, then we necessarily have $\lim_{n\to\infty}k(n)=\infty$ for all possible choices of $k(n)$.
Nevertheless, in the case $\omega(t)=t$ (which gives raise to the important class of $C^{1,1}$ functions), J.C. Wells \cite{Wells} and other authors \cite{LeGruyer1, DaniilidisHaddouLeGruyerLey, AzagraLeGruyerMudarraExplicitFormulas} showed, by very different means, that the $C^{1,1}$ version of the Whitney extension theorem holds true if we replace $\R^n$ with any Hilbert space and, moreover, there is a (nonlinear) extension operator $(f,G)\mapsto (F, \nabla F)$ which is minimal, in the following sense. Given a Hilbert space $X$, with norm denoted by $|\cdot|$, a subset $E$ of $X$, and functions $f:E\to\R$, $G:E\to X$, a necessary and sufficient condition for the $1$-jet $(f,G)$ to have a $C^{1,1}$ extension $(F, \nabla F)$ to the whole space $X$ is that
\begin{equation}\label{GlaeserLeGruyersCondition}
A(f,G):=\sup_{x\in X; \, y, z\in E, \, |x-y| + |x-z|> 0}\frac{ 2|f(y)+\langle G(y), x-y\rangle- f(z)-\langle G(z), x-z\rangle | }{ |x-y|^2 +|x-z|^2 } <\infty.
\end{equation}
Moreover, the extension $(F, \nabla F)$ can be taken with best Lipschitz constants, in the sense that
\begin{equation}\label{in the C11 case the trace seminorm is equal to M}
A(F, \nabla F) = A(f, G) = \|(f,G)\|_{E},
\end{equation}
where
\begin{equation}\label{C11 trace seminorm}
\|(f,G)\|_{E}:= \inf \lbrace \lip(\nabla H) \: : \: H \in C^{1,1}(X) \:\: \text{and} \:\: (H,\nabla H) = (f,G) \:\: \text{on} \:\: E \rbrace
\end{equation}
is the {\em $C^{1,1}$ trace seminorm} of the jet $(f,G)$ on $E$.
In particular, considering $X=\R^n$ we deduce the remarkable corollary that in the case $\omega(t)=t$ one can take $k(n)=1$ for all $n$ in Theorem \ref{WhitneyGlaeserThm}.

Let us point out that condition \eqref{GlaeserLeGruyersCondition} appears in Le Gruyer's paper \cite{LeGruyer1}. Wells' Theorem was stated and proved in \cite{Wells} with the following equivalent condition: there exists a number $M>0$ such that
$$
f(z) \leq f(y) + \frac{1}{2} \langle G(y)+G(z), z-y \rangle + \frac{M}{4} |y-z|^2 - \frac{1}{4M} |G(y)-G(z)|^2 \eqno(W^{1,1})
$$
for all $y,z\in E$. That this condition is equivalent to \eqref{GlaeserLeGruyersCondition} can be easily checked as follows: for each $M>0$ consider the quadratic function
$$
V_M(x):= f(y)+\langle G(y), x-y\rangle -f(z)-\langle G(z), x-z\rangle +\frac{M}{2}\left( |x-z|^2 +|x-y|^2\right),
$$
and find the point $x_M\in X$ that minimizes $V_M$. Then we have $A(f, G)\leq M<\infty$ if and only if $V_{M}(x_{M})\geq 0$, which after a straightforward computation is easily seen to be equal to condition $(W^{1,1})$.

We should also mention that Wells's proof \cite{Wells} was rather elaborate (and constructive only in the case of a finite set $E$), and that Le Gruyer's proof \cite{LeGruyer1}, though very elegant and simple, was not constructive either (Zorn's lemma was used in an essential part of the argument). Very recently, the papers \cite{AzagraLeGruyerMudarraExplicitFormulas, DaniilidisHaddouLeGruyerLey} supplied constructive proofs of Wells' theorem by means of two different explicit formulas, and also provided new proofs (with explicit formulas) for a related $C^{1,1}$ convex extension problem for $1$-jets that had been previously considered in \cite{AzagraMudarra2017PLMS}; see also \cite{AzagraMudarraGlobalGeometry} for the $C^1$ convex case.

For more information about Whitney extension problems for jets and for functions, and the related extension operators, see \cite{BierstoneMilmanPawlucki1, BierstoneMilmanPawlucki2, BrudnyiShvartsman, DacorognaGangbo, Fefferman2005, Fefferman2006, FeffermanSurvey, Fefferman2009, Fefferman2010, FeffermanIsraelLuli3, FeffermanIsraelLuli1, FeffermanIsraelLuli2, FeffermanKlartag, FeffermanShvartsman, Glaeser, HajlaszEtALSobolevExtension, VossHirnMcCollum, Israel, KlartagZobin, KoskelaEtAl,  LeGruyer1, LeGruyer2, Luli, PinamontiSpeightZimmerman, Shvartsman2010, Shvartsman2014, Shvartsman2017, ShvartsmanZobin,  Wells, Whitney, Zobin1998, Zobin1999} and the references therein.

In this paper we will consider the following questions: is Theorem \ref{WhitneyGlaeserThm} true if we replace $\R^n$ with a Hilbert space $X$? Or equivalently, is there a version of Wells's theorem for not necessarily linear moduli of continuity $\omega$? In particular, is Theorem \ref{WhitneyGlaeserThm} true with bounded $k(n)$? And what can be said about other Banach spaces $X$? Let us mention that, as was shown in \cite{JSSG}, a similar question for the class $C^1(X)$ has a positive answer, but to the best of our knowledge nothing is known for nonlinear $\omega$ and the class $C^{1, \omega}(X)$, where $X$ is a Hilbert space (or more generally a Banach space). It is also important to notice that for the classes $C^{k, \omega}(X)$ with $k\geq 2$ this kind of results are not true: the best possible constants $k(n)$ in the higher order versions of Theorem \ref{WhitneyGlaeserThm} established in \cite{Glaeser} must go to $\infty$ as $n\to\infty$; see \cite[Theorem 1 of Section 5]{Wells}.

As we will see, the main result of our paper gives a positive answer to the first question: a jet $(f, G)$ defined on an arbitrary subset $E$ of a Hilbert space $X$ has an extension $(F, \nabla F)$ with $F\in C^{1, \omega}(X)$ if and only if $A(f, G)<\infty$. Moreover, we can take $F$ such that
$$
A(F, \nabla F)\leq 2A(f, G).
$$

In particular, considering $X=\R^n$, this shows that in Theorem \ref{WhitneyGlaeserThm} one can always take $k(n)\leq 2$ for all $n\in\N$.  We will also prove similar results for superreflexive Banach spaces $X$ for a certain class of moduli of continuity.

In order to state and explain our results more precisely, let us introduce some more notation and definitions. Recall that, given a function $g: \R \to \R$, the Fenchel conjugate of $g$ is defined by
\begin{equation}\label{definition of Fenchel conjugate}
g^*(t) = \sup_{s\in \R} \lbrace st -g(s) \rbrace, \quad t\in \R,
\end{equation}
where $g^*$ may take the value $+\infty$ at some $t.$ If $(X, \|\cdot\|)$ is a Banach space, with dual $(X^{*}, \|\cdot\|_{*})$, for any $\xi\in X^{*}$ we let $\langle \xi, v\rangle :=\xi(v)$ denote the duality product.

\begin{definition}
{\em We will say that a $1$-jet $(f,G)$ defined on a subset $E$ of a Banach space $X$ satisfies condition $(W^{1,\omega})$ with constant $M>0$ on $E$ provided that
$$
f(z) \leq f(y) + \frac{1}{2} \langle G(y)+G(z), z-y \rangle + M \varphi(\|y-z\|) - 2M  \varphi^* \left( \frac{\|G(y)-G(z)\|_*}{2M} \right) \eqno (W^{1, \omega})
$$
for all $y,z\in E.$}
\end{definition}

Notice that $\varphi^*$ is simply $\varphi^*(t)=\int_0^t \omega^{-1}(s) ds$ for all $t\geq 0.$ For a mapping $G : E \to X^*$, where $E \subseteq X$, we will denote 
$$
M_\omega(G)= \sup_{x,y\in E, \, x \neq y} \frac{\| G(x)-G(y)\|_*}{\omega(\| x-y \|)}.
$$
\begin{theorem}\label{Mainthm1}
Let $E$ be a nonempty subset of a Hilbert space $(X, |\cdot|)$, and $f: E \to \R, \: G:E \to X$ two functions. There exists $F \in C^{1,\omega}(X)$ such that $(F,\nabla F)=(f,G)$ on $E$ if and only if  $(f,G)$ satisfies $(W^{1,\omega})$ with some $M>0.$ Moreover, we can arrange $M_\omega(\nabla F) \leq (16/\sqrt{3}) M.$
\end{theorem}

On the other hand, for any function $F\in C^{1,\omega}(X),$ the jet $(F, \nabla F)$ satisfies $(W^{1,\omega})$ with constant $M=M_\omega(\nabla F);$ see Proposition \ref{equivalence (mg)(W)}(2) below for a proof. Consequently, if $M^*$ denotes the infimum of those numbers $M>0$ for which $(f,G)$ satisfies $(W^{1,\omega})$ with constant $M,$ Theorem \ref{Mainthm1} yields the following estimate
$$
M^* \leq \inf \lbrace M_{\omega}(\nabla H) \, : \, H \in C^{1,\omega}(X) \:\: \text{and} \:\: (H,\nabla H) = (f,G) \:\: \text{on} \:\: E \rbrace \leq (16/\sqrt{3}) M^*.
$$
We may obtain slightly better constants in the estimate of the gradient if we consider the following extension condition.
\begin{definition}
{\em We will say that a $1$-jet $(f, G)$ defined on a subset $E$ of a Banach space $X$ satisfies condition $(mg^{1, \omega})$ with constant $M$ on $E$ provided that 
$$
f(y)+\langle G(y), x-y\rangle +M\varphi(\|x-y\|) \geq f(z)+\langle G(z), x-z\rangle -M\varphi(\|x-z\|) \eqno(mg^{1, \omega})
$$
for all $y, z\in E$ and all $x\in X$.}
\end{definition}

Thus $(f, G)$ satisfies $(mg^{1, \omega})$ for some $M>0$ if and only if $A(f, G)<\infty,$ and $A(f,G)$ is precisely the smallest $M$ for which $(f,G)$ satisfies $(mg^{1, \omega})$ with constant $M.$ Condition $(mg^{1, \omega})$ is half-intrinsic and half-extrinsic (in what refers to points $x\in X$), as opposed to $(W^{1,\omega})$, which is completely intrinsic (it only concerns points $y, z\in E$). In principle condition $(W^{1, \omega})$ should be easier to check, but conditions like $(mg^{1, \omega})$ may also appear very naturally in some applications (see, for instance, the paper \cite{AzagraHajlasz} in the convex setting). Anyhow both conditions are useful and in fact they are equivalent up to an absolute factor; see Proposition \ref{equivalence (mg)(W)} below. In the case of a nonlinear modulus of continuity $\omega$, these conditions, though equivalent, are no longer identical. This is due to the fact that the minimization of the function
$$
V_M(x):= f(y)+\langle G(y), x-y\rangle -f(z)-\langle G(z), x-z\rangle +M\varphi\left( |x-z|\right) +M\varphi\left(|x-y|\right)
$$
leads us in this case to rather perplexing equations which are difficult to handle and solve. Therefore a condition of the type $V_{M}(x_M)\geq 0$ would be much more complicated than $(W^{1, \omega})$. 

With this {\em extrinsic} condition we have the following.

\begin{theorem}\label{Mainthm2}
Let $E$ be a nonempty subset of a Hilbert space $X$, and $f: E \to \R, \: G:E \to X$ two functions. There exists $F \in C^{1,\omega}(X)$ such that $(F,\nabla F)=(f,G)$ on $E$ if and only $(f,G)$ satisfies $(mg)^{1, \omega}$ for some $M>0$. Moreover, we can arrange $M_\omega(\nabla F) \leq (16/\sqrt{15}) M$.
\end{theorem}

The proof of the preceding theorem also gives us the following nearly optimal result.

\begin{theorem}\label{MainthmOptimalmg}
A $1$-jet $(f, G)$ defined on a nonempty subset $E$ of a Hilbert space $X$ has an extension $(F, \nabla F)$ with $F\in C^{1, \omega}(X)$ if and only if $A(f, G)<\infty$. Moreover, we can take $F$ such that
$$
A(F, \nabla F)\leq 2 A(f, G),
$$
where $A(f, G)$ is defined by \eqref{Our optimal condition}. 
\end{theorem}

For H{\"o}lder moduli of continuity, i.e., when $\omega(t)=t^{\alpha}$ with $\alpha\in (0,1]$, we can improve Theorem \ref{MainthmOptimalmg} as follows.

\begin{theorem}\label{MainthmOptimalmgForSpecialw}
A $1$-jet $(f, G)$ defined on a nonempty subset $E$ of a Hilbert space $X$ has an extension $(F, \nabla F)$ with $F\in C^{1, \alpha}(X)$ if and only if $A(f, G)<\infty$. Moreover, we can take $F$ such that
$$
A(F, \nabla F)\leq 2^{1-\alpha} A(f, G),
$$
where $A(f, G)$ is defined by \eqref{Our optimal condition}. 
\end{theorem}

Note that in the particular case that $\alpha=1$ this result yields Wells' theorem.

\medskip

According to Theorem \ref{Mainthm2} we always have
\begin{equation}\label{estimation of the trace seminorm in general}
\inf \lbrace M_{\omega}(\nabla H) \, : \, H \in C^{1,\omega}(X) \:\: \text{and} \:\: (H,\nabla H) = (f,G) \:\: \text{on} \:\: E \rbrace \leq (16/\sqrt{15}) A(f, G),
\end{equation}
and, in the special case that $\omega(t)=t^\alpha,$ we will see that this estimate can be improved as follows:
\begin{equation}\label{estimation of the trace seminorm w-1tt concave}
\inf \lbrace M_{\omega}(\nabla H) \, : \, H \in C^{1,\omega}(X) \:\: \text{and} \:\: (H,\nabla H) = (f,G) \:\: \text{on} \:\: E \rbrace \leq \frac{2^{2-2\alpha}}{\sqrt{1+\alpha}} \left( 1+\frac{1}{\alpha}\right)^{\alpha/2} A(f, G).
\end{equation}
On the other hand, for any extension $(H, \nabla H)$ of $(f, G)$ with $H\in C^{1, \omega}(X)$ we always have the trivial estimate
$$
A(f, G)\leq M_{\omega}(\nabla H).
$$

Hence we may conclude the following.
\begin{corollary}\label{Various estimates corollary}
A $1$-jet $(f, G)$ defined on a nonempty subset $E$ of a Hilbert space $X$ has an extension $(F, \nabla F)$ with $F\in C^{1, \omega}(X)$ if and only if $A(f, G)<\infty$, in which case we have
$$
A(f, G)\leq \|(f,G)\|_{E, \omega}\leq (16/\sqrt{15}) A(f, G),
$$
where
$$
\|(f, G)\|_{E, \omega}:= \inf \lbrace M_{\omega}(\nabla H) \, : \, H \in C^{1,\omega}(X) \:\: \text{and} \:\: (H,\nabla H) = (f,G) \:\: \text{on} \:\: E \rbrace
$$
and $A(f, G)$ is defined by \eqref{Our optimal condition}.

Furthermore, if $\omega(t)=t^\alpha$ with $0<\alpha \leq 1,$ then we have
$$
A(f, G)\leq \|(f,G)\|_{E, \omega}\leq \frac{2^{2-2\alpha}}{\sqrt{1+\alpha}} \left( 1+\frac{1}{\alpha}\right)^{\alpha/2} A(f, G).
$$
\end{corollary}

It should be noted that for every function of class $F \in C^{1,1}(X)$ defined on a Hilbert space, we always have the identity $\lip(\nabla F) =A(F, \nabla F),$ but this is no longer true for the class $C^{1,\omega}.$ For instance, it is easy to see that the function $f(x)= \frac{2}{3} |x| ^{3/2}, \, x\in \R,$ satisfies $M_\omega(f') = \sqrt{2}$ for $\omega(t)=t^{1/2}.$ However, one can check that $A(f,f') < \sqrt{2}.$ Indeed, since $f$ is $C^1(\R),$ it is clear that 
$$
A(f,f')= \sup_{x,y\in \R, \, x \neq y} \frac{|f(x)-f(y)-f'(y)(x-y)|}{\varphi(|x-y|)} \, , \quad \text{where} \quad \varphi(t)= \tfrac{2}{3} t^{3/2}. 
$$
This supremum is attained at couples of points $(x,y)$ with $x<0<y,$ and, using the homogeneity of $f$ and $f',$ it is not difficult to check that it is equal to
$$
\sup_{t>0} \frac{t^{3/2}+3t^{1/2}+2}{2(t+1)^{3/2}} \leq 1,3066.
$$

\medskip

We will also prove that Theorem \ref{Mainthm2} extends to the class of superreflexive spaces: if $X$ is such a Banach space, thanks to Pisier's results (see \cite[Theorem 3.1]{Pisier}), we can find an equivalent norm $\|\cdot\|$ in $X$ such that may assume that the norm $\| \cdot \|$ is uniformly smooth with modulus of smoothness of power type $p=1+\alpha$ for some $0<\alpha \leq 1$. Hence there exists a constant $C > 0$, depending only on this norm, such that 
\begin{equation}\label{modulussmoothnesssuperreflexive}
 \lambda \| x \|^{1+\alpha} + (1-\lambda) \| y\|^{1+\alpha}- \| \lambda x+(1-\lambda)y\|^{1+\alpha} \leq \lambda (1-\lambda) 2^{-\alpha} C \| x-y\|^{1+\alpha} 
\end{equation}
for all $x,y\in X$, $\lambda \in [0,1]$. 
In particular, we have 
\begin{equation}\label{modulussmoothnesssuperreflexivemidpoints}
 \| u+h \|^{1+\alpha} +  \| u-h\|^{1+\alpha}- 2 \| u\|^{1+\alpha} \leq  C \| h\|^{1+\alpha} \quad \text{for all} \quad u,h\in X.
\end{equation}
We will consider modulus of continuity $\omega$ such that the function $t \mapsto t^\alpha/\omega(t)$ is non-decreasing, which includes the cases $\omega(t)=t^\beta,$ with $\beta \leq \alpha.$ We will then show that an inequality similar to \eqref{modulussmoothnesssuperreflexive} holds true with $\psi_\omega = \varphi_\omega \circ \| \cdot \|$ instead of $\| \cdot \|^{1+\alpha}$, where $\varphi_\omega(t)=\int_0^t \omega(s) ds.$ As a consequence, we will obtain the following theorem in terms of conditions $(mg^{1,\omega}).$ 

\begin{theorem}\label{Mainthm3mgSuperreflexive}
Let $X$ be a superreflexive Banach space with an equivalent norm $\| \cdot \|$ satisfying \eqref{modulussmoothnesssuperreflexive} and let $\omega$ be a modulus of continuity such that $t \mapsto t^\alpha/ \omega(t)$ is non-decreasing. Let $E \subset X$ be a nonempty subset and $f: E \to \R, \: G:E \to X^{*}$ two functions. There exists $F \in C^{1,\omega}(X)$ such that $(F, DF)=(f,G)$ on $E$ if and only $(f,G)$ satisfies $(mg^{1, \omega})$ for some $M>0$. Moreover, we can arrange that $M_\omega(DF) \leq 3 \left(1+  \tfrac{ 3^{1+\alpha}}{1+\alpha} C\right) M.$
\end{theorem}

And if we consider the intrinsic conditions $(W^{1,\omega})$ we have the following.

\begin{theorem}\label{Mainthm4mgSuperreflexive}
Let $X$ be a superreflexive Banach space with an equivalent norm $\| \cdot \|$ satisfying \eqref{modulussmoothnesssuperreflexive} and let $\omega$ be a modulus of continuity such that $t \mapsto t^\alpha/ \omega(t)$ is non-decreasing. Let $E \subset X$ be a nonempty subset and $f: E \to \R, \: G:E \to X^{*}$ two functions. There exists $F \in C^{1,\omega}(X)$ such that $(F, DF)=(f,G)$ on $E$ if and only $(f,G)$ satisfies $(W^{1, \omega})$ for some $M>0$. Moreover, we can arrange that $M_\omega(DF) \leq 12 \left(1+  \tfrac{ 3^{1+\alpha}}{1+\alpha} C\right) M.$ 
\end{theorem}

It is worth noting that the proofs of Theorems \ref{Mainthm3mgSuperreflexive} and \ref{Mainthm4mgSuperreflexive} show that the sufficiency parts of these results still hold true for moduli $\omega$ not necessarily satisfying that the function $t \mapsto t^\alpha/\omega(t)$ is non-decreasing, if we only assume that the function $\psi_{\omega}:=\varphi_{\omega}\circ\|\cdot\|$ is of class $C^{1, \omega}$. However, such an assumption implies superreflexivity of the space $X$ (see \cite[Theorem V.3.2]{DevilleGodefroyZizler}), hence also the existence of an equivalent norm with modulus of smoothness of power type $p=1+\alpha$ for some $\alpha\in (0, 1]$. Let us also mention that in \cite[Section 6]{AzagraLeGruyerMudarraExplicitFormulas} it was shown that a necessary condition on a Banach space $X$ for the validity of a Whitney-type extension theorem in $X$ for some class $C^{1,\omega}$ is that $X$ is superreflexive.

\medskip

Let us finish this introduction by making a few comments on our method of proof and honoring the title of this paper (where we promised some formulas). If one tries to adapt the proof of Wells' theorem given in \cite{AzagraLeGruyerMudarraExplicitFormulas} to the $C^{1, \omega}$ situation, one sees that the argument breaks down for the following reason: when $\omega(t)$ is not linear, it is no longer true that a function $u$ is of class $C^{1, \omega}$ if and only if there exists a convex function $\psi$ of class $C^{1, \omega}$ such that $u+\psi$ is convex and $u-\psi$ is concave. As it turns out, the appropriate class of functions for tackling this more general problem seems to be not that of convex functions, but that of {\em strongly $\varphi$-paraconvex} functions, see Definition \ref{definition of strongly paraconvex} below. 

The main ideas of the proof of Theorem \ref{MainthmOptimalmg} are the following: if $A(f, G)<\infty$ then the functions
\begin{align*}
m(x)= \sup_{z\in E} \lbrace f(z)+\langle G(z),x-z \rangle - M \varphi(|x-z|) \rbrace, \quad x\in X \\
g(x)= \inf_{y\in E} \lbrace f(y)+\langle G(y),x-y \rangle + M \varphi(|x-y|) \rbrace, \quad x\in X
\end{align*}
are well defined and satisfy
$$
m(x)\leq  g(x) \textrm{ for all } x\in X, \textrm{ and } m(y)=g(y)=f(y) \textrm{ for all } y\in E.
$$
Then one can check that the functions $m$ and $(-g)$ are strongly $2M \varphi$-paraconvex and define $F:X \to \R$ by
\begin{equation}\label{first definition of F}
F(x):= \sup\lbrace h(x) \, : \, h \leq g \: \textrm{ and } \: h \: \textrm{ is strongly } 2M \varphi\text{-paraconvex} \rbrace, \quad x\in X.
\end{equation}
One may call $F$ the {\em $2M \varphi$-strongly paraconvex envelope} of $g$. 
As we will show, both $F$ and $-F$ are strongly $2M \varphi$-paraconvex, and this implies that $F\in C^{1,\omega}(X)$ with
$
A(F, \nabla F)\leq 2 A(f, G),
$
and that $\nabla F=G$ on $E$. 

It is also worth noting that, in the very particular case $\omega(t)=t$, one can also define $F$ above with $2$ replaced with $1$. In this special case, another expression for $F$ is the following: for each $t\in\R, p\in X$, $\xi\in X^{*}$, set
\begin{equation}\label{definition of Hptxi}
H_{p, t, \xi}(z) :=t+\langle\xi, z-p\rangle -\frac{M}{2}|z-p|^2.
\end{equation}
Then we have
\begin{equation}\label{second definition of F}
F(x)=\sup\left\{ H_{p, t, \xi}(x) \, : \, t\in\R,\: p\in X, \: \xi\in X^{*}, \:  H_{p, t, \xi}(z)\leq g(z) \: \textrm{ for all } \: z\in X\right\}; 
\end{equation}
see Lemma \ref{another sup formula for F} below.
From this formula we can see that, in the case that $E$ is finite, say that $E$ has $m$ points, then for each $x\in X$, $F(x)$ can be computed by solving a maximization problem in $\R\times X\times X^{*}$ with $m$ constraints, where the function to be maximized and the constraining functions are linear combinations of bilinear functions and quadratic functions. Hence the computation of $F(x)$ is much easier than in the general case of a nonlinear modulus $\omega$. 

When $\omega(t)$ is not necessarily linear, we may also provide an alternate formula for an admissible extension $F$ of $(f, G)$ as the supremum of a smaller family of functions than that of \eqref{first definition of F}: given a $1$-jet $(f, G):E\to\R\times X$ such that $M:=A(f, G)<\infty$, let us define
$$
\mathcal{F} := \left\lbrace X \ni z \mapsto a+ \langle \xi, z \rangle - \sum_{i=1}^n \lambda_i M \varphi(|z-p_i|) \, : \, a\in \R,\, \xi \in X^*, \, p_i\in X, \, \lambda_i \geq 0,\, \sum_{i=1}^n \lambda_i=1, \, n\in \N \right\rbrace,
$$
and
\begin{equation}\label{third definition of F}
F(x) :=\sup\lbrace h(x) \, : \, h\in \mathcal{F}, \: h \leq g \rbrace.
\end{equation}
Then $F$ is of class $C^{1, \omega}(X)$ and satisfies $(F, \nabla F)=(f, G)$ on $E$, with
$$
A(F, \nabla F)\leq 2 A(f, G).
$$
In the case that $\omega(t)$ is linear, it is easily seen that this extension $F$ coincides with \eqref{second definition of F}, and also with $\textrm{conv}(g+ \psi) - \psi;$ where $\psi= \frac{M}{2} | \cdot |^2$ and $\textrm{conv}(g+ \psi)$ denotes the convex envelope of $g+ \psi,$ that is, the supremum of all lower semicontinuous convex functions lying below $g+ \psi.$ This is a consequence of the fact that a function $h: X \to \R$ is strongly $\varphi$-paraconvex if and only if $h+ \psi$ is convex; where $\varphi(t)= \frac{M}{2}t^2$ (however, this is no longer true for nonlinear moduli of continuity).

\medskip

These results will all be shown in Section \ref{sectionproofmainresults} below. In Sections \ref{sectionboundedcase}, \ref{sectionlipschitzcase}  we will give some variants of our techniques which will allow us to establish similar results for the subclasses of $C^{1, \omega}(X)$ consisting of bounded and/or Lipschitz functions, and also a certain continuous dependence of the extensions on the initial data, meaning that if a sequence $\{(f_n, G_n)\}_{n\in\N}$ of jets converges uniformly on $E$ to a jet $(f, G)$ then the corresponding extensions satisfy that $\lim_{n\to\infty}(F_n, \nabla F_n)= (F, \nabla F)$ uniformly on $X$. Finally, in Section \ref{sectionc1ub} we will consider the class $C^{1, u}_{\mathcal{B}}(X)$ of differentiable functions whose derivatives are uniformly continuous on  bounded subsets of $X$, and we will show the following result: suppose that the jet $(f, G)$ is bounded on each bounded subset of $E$; then there exists $F\in C^{1, u}_{\mathcal{B}}(X) $ with $(F, \nabla F)_{|_E}=(f, G)$ if and only if, for all bounded sequences $(x_n)\subset X$, $(y_n), (z_n)\subset E$ with  $|x_n-y_n|+|x_n-z_n|>0$,
$$
\lim_{n\to\infty} \left( |x_n-y_n|+|x_n-z_n| \right)=0 \implies \lim_{n\to\infty}\frac{f(y_n)+\langle G(y_n), x_n-y_n\rangle-f(z_n)-\langle G(z_n), x_n-z_n\rangle}{|x_n-y_n|+|x_n-z_n|}=0.
$$
Also note that, in the particular case $X=\R^n$, we have $C^{1}(\R^n)=C^{1, u}_{\mathcal{B}}(\R^n)$, and this statement is thus equivalent to Whitney's extension theorem for $C^1$. 

\section{Some technical tools}

Recall that the Fenchel conjugate of a function $g$ is denoted by $g^{*}$ and defined as in \eqref{definition of Fenchel conjugate}.

\begin{proposition}\label{elementarypropertiesconjugate} The following properties hold.
\begin{enumerate}
\item $(ag)^*= ag^*( \frac{\cdot}{a})$ and $\left( ag( \frac{\cdot}{a}) \right)^*= a g^*$ for $a>0.$ 
\item  $ab \leq g(a)+g^*(b)$ for $a,b \geq 0.$ 
\item If $(X, \| \cdot \|)$ is a Banach space and $g: \R \to [0,\infty)$ is even, then $\left( g \circ \| \cdot \| \right)^*= g^* \circ \| \cdot \|.$ Here, for a function $\psi:X \to \R,$ we denote $\psi^*(x^*)= \sup_{x\in X} \lbrace \langle x^*, x \rangle - \psi(x) \rbrace$ for every $x^*\in X^*.$
\end{enumerate} 
\end{proposition}

Abusing terminology, we will consider the Fenchel conjugate of nonnegative functions only defined on $[0,+\infty),$ say $\delta: [0,+\infty) \to [0,+\infty)$. In order to avoid problems, we will assume that all the functions involved are extended to all of $\R$ by setting $\delta(t)= \delta(-t)$ for $t<0$. Hence $\delta$ will be an even function on $\R$ and therefore
$$
\delta^*(t)= \sup_{ s\in \R} \lbrace ts-\delta(s) \rbrace =  \sup_{ s \geq 0} \lbrace ts-\delta(s) \rbrace ,\quad \text{for} \quad t \geq 0.
$$

In the following proposition we collect some elementary facts concerning the functions $\omega, \omega^{-1}, \varphi$ and $\varphi^{*}$.

\begin{proposition}\label{relationvarphivarphistar}
Let $\omega:[0, \infty)\to [0,\infty)$ be a concave and nondecreasing function with $\omega(0)=0,$ and define $\varphi(t)=\int_{0}^{t}w(s)ds$. Then:
\begin{enumerate}
\item  $\varphi$ is convex;
\item  $(t/2) \omega(t) \leq \varphi(t) \leq t\omega(t/2)$ for all $t\geq 0$;
\item  $\omega(ct)\leq c\omega(t)$ for all $c\geq 1, t\geq 0.$
\end{enumerate}
If, in addition, $\omega$ is increasing and $\lim_{t\to\infty}\omega(t)=\infty,$ then $\omega^{-1}$ and $\varphi^*$ are well defined and
\begin{enumerate}
\item[$(4)$]  $\varphi^*(t)= \int_0^t \omega^{-1}(s) ds$ for all $t\geq 0$;
\item[$(5)$]  $\varphi(t)+\varphi^*(s)=ts$ if and only if $s= \omega(t);$ 
\item[$(6)$] $t\omega^{-1}(t/2)\leq \varphi^{*}(t)\leq (t/2)\omega^{-1}(t)$ for all $t\geq 0$.
 \end{enumerate}
\end{proposition}

\begin{lemma}\label{constantnormomega}
Let $(X, |\cdot|)$ be a Hilbert space, and $\omega$ a modulus of continuity as in the preceding proposition. Then the function $\psi(x)= \varphi(| x|), \: x\in X$, satisfies the following inequality:
$$
\psi( \lambda x+(1-\lambda) y) \geq \lambda \psi(x)+ (1-\lambda)\psi(y) - 2 \lambda(1-\lambda)\varphi(|x-y|) \quad \text{for all} \quad x,y \in X, \: \lambda \in [0,1].
$$
\end{lemma}
\begin{proof}
Since $\varphi^*(t)= \int_0^t \omega^{-1}(s)ds,$ by Proposition \ref{elementarypropertiesconjugate}, the Fenchel conjugate $\psi^*$ of $\psi$ satisfies $\psi^*(x)=\int_0^{|x|} \omega^{-1}(s)ds$, where $\omega^{-1}$ is a convex function. We know from \cite{VladimirovNesterovCekanov} that $\psi^*$ is uniformly convex on $X$ with modulus of convexity $\delta(t)= \int_0^t \omega^{-1}(s/2)ds,$ that is,
$$
\psi^*( \lambda x+(1-\lambda) y) \leq \lambda \psi^*(x)+ (1-\lambda)\psi^*(y) -  \lambda(1-\lambda)\delta(|x-y|) \quad \text{for all} \quad x,y \in X, \: \lambda \in [0,1].
$$
Using the duality theorem (see \cite[Proposition 3.5.3]{Zalinescubook}, for instance), we obtain that $\psi= (\psi^*)^*$ is uniformly smooth with modulus of smoothness $\delta^*,$ that is,
$$
\psi( \lambda x+(1-\lambda) y) \geq \lambda \psi(x)+ (1-\lambda)\psi(y) -  \lambda(1-\lambda)\delta^*(|x-y|) \quad \text{for all} \quad x,y \in X, \: \lambda \in [0,1].
$$
We have that $\delta(t)=\int_0^t k(s) ds,$ where $k(s)= \omega^{-1}(s/2)$ and therefore 
$$
\delta^*(t)=\int_0^t k^{-1}(s)ds = \int_0^t 2 \omega(s)ds =2\varphi(t). 
$$
\end{proof}

For H\"older moduli of continuity, the preceding lemma is true with constant $2^{1-\alpha}$ instead of $2$. 

\begin{lemma}\label{constantnormholder}
Let $(X, |\cdot|)$ be a Hilbert space, and $\omega(t)=t^\alpha$ for $\alpha\in (0,1].$ Then the function $\psi(x)= \varphi(| x|), \: x\in X$, satisfies the following inequality:
$$ 
\psi( \lambda x+(1-\lambda) y) \geq \lambda \psi(x)+ (1-\lambda)\psi(y) - 2^{1-\alpha} \lambda(1-\lambda)\varphi(|x-y|) \quad \text{for all} \quad x,y \in X, \: \lambda \in [0,1].
$$
\end{lemma}
\begin{proof}
The Fenchel conjugate $\psi^*$ of $\psi$ satisfies $\psi^*(x)= \int_0^{|x|} t^{1/\alpha} dt=(1+\alpha^{-1})^{-1} |x|^{1+\alpha^{-1}}.$ We know from \cite{VladimirovNesterovCekanov} that $\psi^*$ is uniformly convex with modulus of convexity 
$$
\delta(t)=\frac{2^{1-\alpha^{-1}}}{1+ \alpha^{-1} }t^{1+\alpha^{-1}}.
$$ 
Thanks to Proposition \ref{relationvarphivarphistar} we have 
$$
\delta^*(t)=  2^{1-\alpha^{-1} } \left( s \mapsto \frac{s^{1+\alpha^{-1}}}{1+ \alpha^{-1}} \right)^*\left(  2^{\alpha^{-1}-1} t \right) = 2^{1-\alpha^{-1} } \left( s \mapsto \frac{s^{1+\alpha}}{1+ \alpha} \right)\left(  2^{\alpha^{-1}-1} t \right) = \frac{2^{1-\alpha}}{1+\alpha}t^{1+\alpha}.
$$
By using the duality theorem as in Lemma \ref{constantnormomega}, we obtain the desired inequality. 
\end{proof}

\medskip
\begin{definition}\label{definition of strongly paraconvex}
{\em If $C\geq 0$ is a constant, we will say that a function $u$ is strongly $C\varphi$-paraconvex on a Banach space $X$ if we have
\begin{equation}\label{strong phi paraconvexity}
u\left( tx+ (1-t)y\right) \leq t u\left( x\right) +(1-t) u \left( y\right)+C t(1-t)\varphi\left( \|x-y \|\right)
\end{equation}
for all $x, y\in X$ and all $t\in [0, 1]$. }
\end{definition}
Thus the preceding two lemmas can be restated by saying that $-\psi$ is strongly $C\varphi$-paraconvex for some $C>0$. On the other hand, since $\psi$ is also convex, $\psi$ is trivially strongly $\varphi$-paraconvex.

Some authors call such functions $u$ {\em semiconvex}, or $\varphi$-semiconvex, but we prefer not to use this terminology because it may make the reader think that the function $u+ C\varphi\left(|\cdot|\right)$ will be convex, at least locally for some large $C$, which is generally false unless $\omega$ is linear. See \cite{CannarsaSinestrari, JouraniThibaultZagrodny, Rolewicz1, Rolewicz2} and the references therein for background on paraconvex and strongly $\varphi$-paraconvex functions.

Next we recall a well-known fact about this kind of functions which we will have to use in our proofs. This result is usually shown in more specialized settings with the help of subdifferentials or Clarke's generalized gradients. For the reader's convenience (and also because we need precise estimates and the literature's terminology varies depending on authors), we include a self-contained elementary proof of this result.

\begin{proposition}\label{if u and -u are strongly paraconvex then u is unoformly differentiable}
Let $(X, \|\cdot\|)$ be a Banach space, $\omega$ a modulus of continuity, $\varphi(t)=\int_{0}^{t}\omega(s)ds$, and $u:X\to\R$ be a continuous function. Assume that both $u$ and $-u$ are strongly $C\varphi$-paraconvex. Then $u$ is everywhere Fr\'echet differentiable, and, with the notation of \eqref{Our optimal condition}, $A(u,Du) \leq C.$ In particular $u$ is of class $C^{1,\omega}(X)$ with
$$
\|Du(x)-Du(y)\|_* \leq \frac{2C\varphi\left( \frac{3}{2}\|x-y \| \right)}{\|x-y\|} \leq 3 C \omega(\|x-y\|)
$$
for all $x, y\in X.$ Moreover, if $X$ is a Hilbert space, we have 
$$
| D u (x)-Du(y) | \leq  C \min \left\lbrace  \frac{8}{\sqrt{15}} \, \omega(|x-y|) , \frac{4}{\sqrt{3}} \, \omega \left( \frac{|x-y|}{2} \right) \right\rbrace \quad \text{for all} \quad x,y\in X,
$$
and $| D u (x)-Du(y) | \leq   \frac{2^{1-\alpha}}{\sqrt{1+\alpha}} \left( 1+\frac{1}{\alpha}\right)^{\alpha/2}  C \, |x-y|^\alpha$ in the special case that $\omega(t)=t^\alpha,\, \alpha \in (0,1].$ 
\end{proposition}
\begin{proof}
Taking $y=a$ and $h=x-a$ in \eqref{strong phi paraconvexity} we see that $u$ satisfies
\begin{equation}\label{strong phi paraconvexity with a h}
u(a+th)\leq t u(a+h)+ (1-t) u(a) +Ct(1-t)\varphi\left(\|h\|\right)
\end{equation}
for all $a, h\in X$, $t\in [0,1]$. Also (taking $t=1/2$, $2a=x+y$, $2h=x-y$) we have
$$
u(a+h)+u(a-h)-2u(a)\geq -\tfrac{C}{2}\varphi(2\|h\|),
$$
and since $-u$ is strongly $C\varphi$-paraconvex too, we obtain
\begin{equation}\label{differentiability with 2u(a)}
|u(a+h)+u(a-h)-2u(a)|\leq \tfrac{C}{2}\varphi(2\|h\|).
\end{equation}
For the moment, let us fix $a$ and $h$ in $X,$ and consider $s,t\in (0,1]$. The inequality \eqref{strong phi paraconvexity with a h} implies
$$
\frac{u(a+tsh)-u(a)}{ts}\leq \frac{u(a+sh)-u(a)}{s}+ C(1-t)\frac{\varphi(s\|h\|)}{s}.
$$
Similarly, because $-u$ is also strongly $C\varphi$-paraconvex, we have
$$
 \frac{u(a+sh)-u(a)}{s}\leq \frac{u(a+tsh)-u(a)}{ts}+ C(1-t)\frac{\varphi(s\|h\|)}{s}.
$$
Therefore
\begin{equation}\label{inequality for differentiability of u}
\left|  \frac{u(a+sh)-u(a)}{s} -\frac{u(a+tsh)-u(a)}{ts} \right|
\leq C(1-t)\frac{\varphi(s\|h\|)}{s}
\end{equation}
for all $s, t\in (0, 1]$, $a, h\in X$.
This entails the existence and local uniform boundedness of the limit
\begin{equation}\label{existence of directional derivatives of u}
\lim_{t\to 0^{+}}\frac{u(a+tv)-u(a)}{t} :=D_{v}u(a)
\end{equation}
for $a, v\in X$. Indeed, on the one hand, by taking $s=1$ and using that $u$ is locally bounded we see that there is some $r>0$ and a constant $k_r$ such that
\begin{equation}\label{boundedness of the differential quotient for u at a}
\left|\frac{u(a+th)-u(a)}{t}\right|\leq k_r \textrm{ for all } h\in B(0, r).
\end{equation}
On the other hand, if the limit in \eqref{existence of directional derivatives of u} did not exist then there would be some $\varepsilon>0$ and two sequences $(s_n), (r_n)$ of strictly positive numbers converging to $0$ such that
$$
\left| \frac{u(a+s_n v)-u(a)}{s_n}- \frac{u(a+r_n v)-u(a)}{r_n}\right|\geq \varepsilon
$$
for all $n$. Up to extracting subsequences we may assume that $0<r_n<s_n$ for all $n$, and then find $(t_n)\subset (0, 1]$ such that $r_n=t_n s_n$ for every $n$, so that the above inequality reads
$$
\left| \frac{u(a+s_n v)-u(a)}{s_n}- \frac{u(a+t_n s_n v)-u(a)}{t_n s_n}\right|\geq \varepsilon,
$$
in contradiction with \eqref{inequality for differentiability of u} and the fact that
\begin{equation}\label{varphishovers goes to 0}
\lim_{s\to 0^{+}}\frac{\varphi(s\|v\|)}{s}=0.
\end{equation}
Next, by using \eqref{differentiability with 2u(a)} and \eqref{varphishovers goes to 0} we also get
$$
\lim_{t\to 0^{+}}\frac{u(a+tv)-u(a)+u(a-tv)-u(a)}{t}=0,
$$
which shows that $D_{-v}u(a)=-D_{v}u(a)$ and consequently that the directional derivative
$$
\lim_{t\to 0}  \frac{u(a+tv)-u(a)}{t}
$$
exists and equals $D_{v}u(a)$.
Furthermore, by letting $t$ go to $0$ in \eqref{inequality for differentiability of u} we also have
\begin{equation}\label{second inequality for differentiability of u}
\left|  \frac{u(a+sv)-u(a)}{s} -D_{v}u(a) \right|
\leq C\frac{\varphi(s\|v\|)}{s},
\end{equation}
for every $a, v\in X$, $s\in (0,1]$,
and in particular
\begin{equation}\label{third inequality for differentiability of u}
\left| u(a+v)-u(a)-D_{v}u(a) \right|
\leq C \varphi(\|v\|)
\end{equation}
for all $a, v\in X$. 

In order to finish the proof that $u$ is differentiable, we will now combine some calculations from \cite[Theorem 3.3.7]{CannarsaSinestrari} and \cite[Theorem 6.1]{JouraniThibaultZagrodny}. We do not yet know that the function $v\mapsto D_v u(a)$ is linear, but we do easily get that $D_{\lambda v}u(a)=\lambda D_{v}u(a)$ for all $a\in X$ and $\lambda\in\R$; this fact is a straightforward consequence of \eqref{second inequality for differentiability of u} which we will use before establishing the linearity of $v\mapsto D_v u(a)$.
We next show that 
$$
\sup_{\|v\|=1}|D_{v}u(a)-D_{v}u(b)|\leq 5 C\omega(\|a-b\|)
$$
for all $a, b\in X$.
Indeed, writing $b=a+h$ with $h\neq 0$, and using the strong $\varphi$-paraconvexity of $u$ and $-u$, and the fact that $\varphi(t)\leq t\omega(t/2)$, we have
\begin{eqnarray*}
& & D_{v}u(b)-D_{v}u(a)= D_{v}u(a+h)-D_{v}u(a) \\
 & & \leq u(a+h+v)-u(a+h)-u(a+v)+u(a)+2 C \|v\|\omega\left(\tfrac{\|v\|}{2} \right) \\
 & & = u(a+h+v)-\tfrac{1}{2}u(a+2v)-\tfrac{1}{2}u(a+2h) \\
& & \quad - u(a+h)+\tfrac{1}{2}u(a+2h)+\tfrac{1}{2}u(a) \\
& &  \quad -u(a+v)+\tfrac{1}{2}u (a+2v)+\tfrac{1}{2}u(a)+2 C\|v\|\omega\left(\tfrac{\|v\|}{2}\right) \\
& & \leq  \tfrac{C}{2}\|h-v\|\omega\left(\|h-v\|\right) +\tfrac{C}{2}\|h\| \omega\left( \|h\|\right)+\tfrac{C}{2}\|v\| \omega\left( \|v\|\right)+2 C \|v\|\omega\left(\tfrac{\|v\|}{2}\right),
\end{eqnarray*}
which implies
$$
\sup_{\|v\|=1}|D_{v}u(a+h)-D_{v}u(a)|=\frac{1}{\|h\|}\sup_{\|v\|=\|h\|}|D_{v}u(a+h)-D_{v}u(a)| \leq 5 C\omega\left( \|h\| \right).
$$
Observe also that $\sup_{\|v\|\leq 1}|D_{v}u(a)|$ is finite for every $a$, thanks to \eqref{boundedness of the differential quotient for u at a}.
Now we may show that $v\mapsto D_v u(a)$ is linear, which together with \eqref{third inequality for differentiability of u} yields that $u$ is uniformly Fr\'echet differentiable. For every $a, v, w\in X$ we have
\begin{eqnarray*}
& &|u(a+v+w)-u(a)-D_{v}u(a)-D_{w}u(a)| \\
& & \leq |u(a+v+w)-u(a+v)-D_{w}u(a+v)|+|D_{w}u(a+v)-D_{w}u(a)|+|u(a+v)-u(a)-D_v u(a)| \\
& & \leq C \varphi\left(\|w\|\right)+5 C\|w\|\omega\left(\|v\|\right)+C\varphi\left( \|v\|\right),
\end{eqnarray*}
from which we easily deduce that $D_{v+w}u(a)=D_{v}u(a)+D_{w}u(a)$. We thus have that $u$ is everywhere Fr\'echet differentiable, and from \eqref{third inequality for differentiability of u} we obtain that the jet $(u,Du): X \to \R \times X^*$ satisfies $A(u, Du)\leq C.$ The estimations for the modulus of continuity of $Du$ are a consequence of Proposition \ref{equivalence (mg)(W)}(3) below.
\end{proof}

\medskip

Let us finish this section by studying what one could fairly call the $C\varphi$-paraconvex envelope of a function.
\begin{definition}
{\em Given a Hilbert space $X$, a continuous function $g: X\to\R$, and a number $C>0$, let us define 
$$
\mathcal{W}_{C}(x)=\sup\{ h(x) \, : \, h\leq g, \, h \textrm{ is strongly } C\varphi\textrm{-paraconvex}\}.
$$  } 
\end{definition}

\begin{lemma}\label{strongly paraconvex functions are uniformly proximal subdifferentiable everywhere}
If $h$ is strongly $C\varphi$-paraconvex on a Hilbert space $X$ then, for each $x\in X$ there exists some $\xi_x\in X^{*}$ such that
$$
h(y)\geq h(x)-\langle\xi_x, y-x\rangle-C\varphi(|y-x|)
$$
for all $y\in X$. In consequence, $h(y)=\sup_{x\in X} \lbrace h(x)-\langle\xi_x, y-x\rangle-C\varphi(|y-x|) \rbrace$ for every $y\in X.$
\end{lemma}
\begin{proof}
Indeed, since $h$ is strongly $C\varphi$-paraconvex, it is locally Lipschitz (see \cite[Proposition 6.1]{JouraniThibaultZagrodny} for a proof of this fact), and then the Clarke subdifferential $\partial_C h(x)$ is nonempty for every $x\in X.$ Moreover, according to \cite[p. 219]{JouraniThibaultZagrodny}, the Clarke subdifferential of $h$ can be written as
$$
\partial_C h(x)= \lbrace \xi \in X^* \, : \,  \langle \xi, v \rangle \leq h(x+v)-h(x) + C \varphi(|v|) \: \text{ for all } \: |v| \leq \delta-|x-a| \rbrace,
$$
for every $x\in B(a,\delta),$ where $B(a,\delta)$ is any ball such that $h_{|_{B(a,\delta)}}$ is Lipschitz. Using that $h$ is strongly $C\varphi$-paraconvex we can prove that, in fact, we have the formula
\begin{equation}\label{clarkesubdifferentialformula}
\partial_C h(x)= \lbrace \xi \in X^* \, : \,  \langle \xi, v \rangle \leq h(x+v)-h(x) + C \varphi(|v|) \: \text{ for all } \: v\in X \rbrace \quad \text{for all} \quad x\in X.
\end{equation}
Indeed, if $x \in X, \: \xi \in \partial_C h (x)$ and $v\in X,$ we consider $t>0$ small enough so that $|tv| \leq \delta,$ where $\delta>0$ is such that $h_{|_{B(a,\delta)}}$ is Lipschitz. Then we have 
$$
h(x+tv)-th(x+v)-(1-t)h(x) \leq t(1-t)C\varphi(|v|),
$$
which implies
$$
\langle \xi, v \rangle \leq t^{-1} \left( h(x+tv)-h(x)+C\varphi(t|v|) \right) \leq h(x+v)-h(x) + (1-t)C \varphi(|v|)+ Ct^{-1}\varphi(t|v|).
$$
Letting $t\to 0^+$ and taking into account that $\varphi(t) \leq t\omega(t)$ and $\lim_{t \to 0^+} \omega(t)=0,$ we get 
$$
\langle \xi, v \rangle \leq h(x+v)-h(x) + C \varphi(|v|).
$$
We have thus shown \eqref{clarkesubdifferentialformula}. Since $h$ is locally Lipschitz we have $\partial_C h(x)\neq\emptyset$ for every $x\in X$ and the result follows.
\end{proof}

\begin{lemma}\label{another sup formula for F}
Assume that $\omega(t)=at$, where $a>0$. Then we have, for every $x\in X$, 
$$
\mathcal{W}_{C}(x)=\sup\left\{ H_{p, t, \xi}(x) \, : \, t\in\R, \: p\in X, \: \xi\in X^{*}, \:  H_{p, t, \xi}(z)\leq g(z) \: \textrm{ for all } \: z\in X\right\},
$$
where
$$
H_{p, t, \xi}(z) :=t+\langle\xi, z-p\rangle -C\varphi\left(|z-p|\right).
$$
\end{lemma}
\begin{proof}
Let us call
$$
H(x):=\sup\left\{ H_{p, t, \xi}(x) \: : \: t\in\R, \: p\in X, \: \xi\in X^{*}, \:  H_{p, t, \xi}(z)\leq g(z) \: \textrm{ for all } \: z\in X\right\}.
$$
On the one hand, by using Lemma \ref{constantnormholder} with $\alpha=1$, we have that $H_{p, t, \xi}$ is strongly $C\varphi$-paraconvex, hence it is clear that
$$
H(x)\leq \mathcal{W}_{C}(x).
$$
On the other hand, if $h$ is strongly $C\varphi$-paraconvex and $h\leq g$ then, according to the previous lemma, there exists some $\xi\in X^{*}$ such that
\begin{equation}
g(y)\geq h(y)\geq h(x)-\langle \xi, y-x\rangle -C\varphi(|y-x|)= H_{x, h(x), \xi}(y)
\end{equation}
for all $y\in X$.
Because $y\mapsto H_{x, h(x), \xi}(y)$ is strongly $C\varphi$-paraconvex and lies below $g$, we have, by definition of $H$,
$$
H(x)\geq H_{x, h(x), \xi}(x)=h(x).
$$
Therefore $H\geq h$ for every $h$ that is strongly $C\varphi$-paraconvex and lies below $g$. Since $\mathcal{W}_{C}$ is the supremum of all such $h$, we also have
$$
\mathcal{W}_{C}(x)\leq H(x) 
$$
for all $x\in X$. Thus we conclude $H=\mathcal{W}_{C}$. 
\end{proof}

\medskip

\section{Proofs of the main results}\label{sectionproofmainresults}

Let us start by showing the equivalence between conditions $(mg^{1, \omega})$ and $(W^{1, \omega})$ and their relation with the quantity $M_\omega(G)$ for jets $(f,G)$ defined on subsets of Banach and Hilbert spaces.

\begin{proposition}\label{equivalence (mg)(W)}
Let $E$ be an arbitrary subset of a Banach space $(X, \|\cdot\|)$, and consider a jet $(f, G):E\subset X\to \R\times X^{*}$.
\begin{enumerate}
\item Assume that $(f, G)$ satisfies condition $(W^{1, \omega})$ with constant $M>0$. Then we have
$$
f(y)+\langle G(y), x-y\rangle +M\varphi(2\|x-y\|) \geq f(z)+\langle G(z), x-z\rangle -M\varphi(2\|x-z\|)
$$
for all $y, z\in E$ and all $x\in X$. In particular $(f, G)$ satisfies $(mg^{1, \omega})$ with constant $4M$.

\item Assume that $(f, G)$ satisfies $(mg^{1,\omega})$ with constant $M$. Then $(f, G)$ satisfies $(W^{1, \omega})$ with constant $2M$. And, if $X$ is a Hilbert space, we can replace $2M$ with $M.$

\item Assume that $(f, G)$ satisfies $(mg^{1,\omega})$ with constant $M$. Then $G$ satisfies
$$
\|G(y)-G(z)\|_*  \leq \frac{2M\varphi\left( \frac{3}{2}\|y-z\| \right)}{\|y-z\|}  \leq 3 M \omega(\|y-z\|), \quad y,z\in E,
$$
and, in particular, $M_\omega(G) \leq 3M.$ Moreover, if $X$ is a Hilbert space, then
$$
|G(y)-G(z)| \leq M \min \left\lbrace  \frac{8}{\sqrt{15}} \, \omega(|y-z|) , \frac{4}{\sqrt{3}} \, \omega \left( \frac{|y-z|}{2} \right) \right\rbrace, \quad y,z\in E.
$$
Furthermore, if $X$ is a Hilbert space and $\omega(t)=t^\alpha, \, \alpha \in (0,1],$ we have $M_\omega(G) \leq \frac{2^{1-\alpha}}{\sqrt{1+\alpha}} \left( 1+\frac{1}{\alpha}\right)^{\alpha/2} M.$

\item If $(f,G)$ satisfies $(W^{1,\omega})$ with constant $M,$ then $G$ is uniformly continuous, with $M_\omega(G) \leq 4M.$

\end{enumerate}
\end{proposition}
\begin{proof}
$(1)$ Let $y,z\in E$ and $x\in X.$ By the assumption we have
$$
f(y) \geq f(z) + \tfrac{1}{2} \langle G(z)+G(y), y-z \rangle - M \varphi(\|z-y\|) + (2M  \varphi)^* \left( \|G(y)-G(z)\|_{*} \right).
$$
Here we have used that $(2M\varphi)^*(t)=2M \varphi^*(\frac{t}{2M});$ see Proposition \ref{elementarypropertiesconjugate}. Employing the preceding inequality we obtain
\begin{align}\label{firstinequalitycomparison}
\nonumber f(y)+&\langle G(y),x-y \rangle  + M \varphi(2\|x-y\|)= f(y)+\langle G(y),z-y \rangle +\langle G(y),x-z \rangle+ M \varphi(2\|x-y\|)  \\
& \nonumber \geq f(z)+ \tfrac{1}{2}\langle G(z)-G(y),y-z \rangle + \langle G(y),x-z \rangle \\
& \nonumber \quad + M \varphi(2\|x-y\|) - M \varphi(\|z-y\|) + (2M  \varphi)^* \left( \|G(y)-G(z)\|_{*} \right) \\
& = f(z)+ \langle G(z), x-z \rangle + \langle G(y)-G(z), x - \tfrac{1}{2}(y+z) \rangle \\
& \nonumber \quad +  M \varphi(2\|x-y\|) - M \varphi(\|z-y\|) + (2M  \varphi)^* \left( \|G(y)-G(z)\|_{*}\right).
\end{align}
By applying the Fenchel-Young inequality (see Proposition \ref{elementarypropertiesconjugate}) for the function $2M \varphi$ we get
\begin{align*}
  \big \langle G(y)-G(z),& \, x - \tfrac{1}{2}(y+z) \big  \rangle \geq - \|G(y)-G(z)\|_{*} \left \| x - \tfrac{1}{2}(y+z) \right \| \geq \\
 & -2M \varphi \left( \left \| x - \tfrac{1}{2}(y+z) \right \| \right)-(2M\varphi)^*(\|G(y)-G(z)\|_{*} ).
\end{align*}
By plugging this inequality into \eqref{firstinequalitycomparison}, we get
\begin{align}\label{secondtinequalitycomparison}
& \nonumber f(y)+\langle G(y),x-y \rangle  + M \varphi(2\|x-y\|) \\
&  \geq f(z)+ \langle G(z), x-z \rangle  -2M \varphi \left( \left \| x - \tfrac{1}{2}(y+z) \right \| \right)+ M \varphi(2\|x-y\|) - M \varphi(\|z-y\|).
\end{align}
Now observe that, since $\varphi$ is convex, we have that $\varphi(t+s) \leq \tfrac{1}{2} \varphi(2t)+ \tfrac{1}{2} \varphi(2s)$ for every $t,s \geq 0.$ This yields
\begin{align*}
-2\varphi & \left( \left \| x - \tfrac{1}{2}(y+z) \right \| \right)+  \varphi(2\|x-y\|) -  \varphi(\|z-y\|) \\
& \geq - \varphi(\|x-y\|)-\varphi(\|x-z\|) +\varphi(2\|x-y\|)-\tfrac{1}{2}\varphi(2\|x-z\|)-\tfrac{1}{2}\varphi(2\|x-y\|) \\
& \geq -\varphi(2\|x-z\|)+\tfrac{1}{2}\varphi(2\|x-y\|)- \varphi(\|x-y\|) \geq -\varphi(2\|x-z\|).
\end{align*}
By plugging this inequality in \eqref{secondtinequalitycomparison}, we deduce
$$
f(y)+\langle G(y),x-y \rangle  + M \varphi(2\|x-y\|) \geq f(z)+ \langle G(z), x-z \rangle  -M\varphi(2\|x-z\|).
$$

\medskip

$(2)$ Let us fix $y,z\in E.$ In the case when $G(y)=G(z),$ condition $(mg^{1,\omega})$ gives
$$
f(z) \leq f(y) +  \langle G(y) ,z-y \rangle + M \varphi(\|y-z\|) 
$$
and the inequality defining condition $(W^{1,\omega})$ with constant $M>0$ follows immediately. Let us assume that $G(y) \neq G(z).$ Given $\varepsilon>0,$ there exists $v=v_\varepsilon \in X$ such that
$$
\|v\| = \tfrac{1}{2}\omega^{-1} \left( \tfrac{\|G(y)-G(z)\|_*}{2M} \right) \quad \text{and} \quad \langle G(z)-G(y) , v \rangle \geq \tfrac{\|G(y)-G(z)\|_*}{2}\omega^{-1} \left( \tfrac{\|G(y)-G(z)\|_*}{2M} \right) -\varepsilon.
$$
Define $x := \tfrac{1}{2}(y+z)+ v.$ Using condition $(mg^{1,\omega})$  we have 
\begin{align}\label{intermediateinequalityintrinsicextrinsicbanach}
f(z) & \leq f(y)+  \langle G(y), x-y \rangle - \langle G(z) , x-z \rangle + M \varphi( \| x-y\|) + M\varphi (\| x-z \| )   \\ 
& = f(y)+ \tfrac{1}{2} \langle G(y)+G(z),  z-y \rangle +\langle G(y)- G(z) , v \rangle + M  \varphi \left( \left \|\tfrac{1}{2}(z-y)+v \right \| \right) +  M \varphi \left(  \left \| \tfrac{1}{2}(y-z)+v  \right \| \right) \notag \\
& \leq f(y)+ \tfrac{1}{2}  \langle G(y)+G(z) , z-y  \rangle +  \langle G(y)- G(z), v \rangle + 2M  \varphi \left( \left \|\tfrac{1}{2}(z-y) \right \|+ \|v\| \right) \notag  \\
& \leq f(y)+ \tfrac{1}{2}  \langle G(y)+G(z),  z-y \rangle +  \langle G(y)- G(z),v \rangle + M \varphi(2\|v\|) +  M  \varphi \left(  \|y-z  \|\right) \notag,
\end{align}
where the last inequality follows from the convexity of $\varphi.$ The identity $t \omega(t)=\varphi(t)+\varphi^*(\omega(t))$ (see Proposition \ref{relationvarphivarphistar}) for $t=2\|v\|$ gives
$$
\varphi(2\|v\|) = 2 \|v\|\omega(2\|v\|) - \varphi^*(\omega(2\|v\|)) = \tfrac{\|G(y)-G(z)\|_*}{2M}\omega^{-1}\left( \tfrac{\|G(y)-G(z)\|_*}{2M} \right) - \varphi^*\left(\tfrac{\|G(y)-G(z)\|_*}{2M} \right).
$$ 
Using this inequality and the properties of $v$ we obtain
\begin{align*}
\langle G(y)- G(z) ,v \rangle + M \varphi(2\|v\|) & \leq -\tfrac{\|G(y)-G(z)\|_*}{2}\omega^{-1} \left( \tfrac{\|G(y)-G(z)\|_*}{2M} \right)+\varepsilon  + M \varphi(2\|v\|) \\
& = -M \varphi^*\left(\tfrac{\|G(y)-G(z)\|_*}{2M} \right) + \varepsilon.
\end{align*}
And observe that $\varphi^*(s) = \int_0^s \omega^{-1}(r) dr = 2 \int_0^{s/2} \omega^{-1}(2r) dr \geq 4 \varphi^*(s/2).$ Combining all these inequalities we get
$$
f(z) \leq  f(y)+ \tfrac{1}{2} \langle G(z)+G(y),  z-y  \rangle +   M  \varphi \left(  \|y-z  \|\right)- 4M \varphi^*\left(\tfrac{\|G(y)-G(z)\|_*}{4M} \right) + \varepsilon.
$$
Letting $\varepsilon \to 0^+,$ we deduce that $(f,G)$ satisfies condition $(W^{1,\omega})$ with constant $2M.$

Let us now modify the proof so as to obtain $M$ in place $2M$ for condition $(W^{1,\omega})$ under the assumption that $X$ is a Hilbert space. Assuming as above $G(y) \neq G(z),$ we define 
$$
x := \tfrac{1}{2}(y+z)+ v, \:  \textrm{ where } \:   v := \frac{\omega^{-1}\left( \frac{| G(y)- G(z) | }{2M} \right)}{| G(y)- G(z) | }( G(z)- G(y)).
$$
If we apply Lemma \ref{constantnormomega} for $\lambda=1/2$ in the third inequality of \eqref{intermediateinequalityintrinsicextrinsicbanach}, we obtain
\begin{align*}
f(z) & \leq f(y) + \tfrac{1}{2} \langle G(y)+ G(z), z-y \rangle + \langle G(y)- G(z), v \rangle + 2M \varphi( | v | ) + M \varphi( | y-z| )  \\
& =f(y) + \tfrac{1}{2} \langle G(y)+ G(z), z-y \rangle +M \varphi( | y-z| ) +2M\left( - s \omega^{-1}(s)+\varphi( \omega^{-1}(s) ) \right),
\end{align*}
where we have denoted $s = \frac{1}{2M} | G(y)-G(z) |.$ By Proposition \ref{relationvarphivarphistar}, we have that $\varphi( \omega^{-1}(s) )-s \omega^{-1}(s) = -\varphi^*(s).$ This immediately implies the inequality defining condition $(W^{1,\omega})$ with constant $M.$

\medskip

$(3)$ Let $y,z\in E$ and $v\in X.$ For the point $x = \frac{1}{2}(y+z)+v,$ condition $(mg^{1,\omega})$ gives
$$
f(z) \leq f(y) + \tfrac{1}{2} \langle G(y)+G(z), z-y \rangle  + \langle G(y)-G(z), v \rangle +   M \varphi\left( \| \tfrac{1}{2}(z-y)+v \| \right) + M \varphi\left( \| \tfrac{1}{2}(y-z)+v \| \right).
$$
Reversing the roles of $y$ and $z,$ and taking $x= \frac{1}{2}(y+z)-v$ in condition $(mg^{1,\omega})$ we obtain
$$
f(y) \leq f(z) + \tfrac{1}{2}  \langle G(y)+G(z),y-z \rangle  + \langle G(y)-G(z), v \rangle +   M\varphi\left( \| \tfrac{1}{2}(y-z)-v \| \right) + M\varphi\left( \| \tfrac{1}{2}(z-y)-v \| \right).
$$
By summing both inequalities we have
\begin{equation}\label{firstestimateMomegaG}
\langle G(z)-G(y) , v \rangle \leq  M\varphi\left( \| \tfrac{1}{2}(z-y)+v \| \right) + M\varphi\left( \| \tfrac{1}{2}(y-z)+v \| \right) \leq 2 M \varphi\left( \tfrac{\|y-z \|}{2} + \|v\| \right).
\end{equation}

These estimates  hold for any $v\in X$, and in particular for every $v\in X$ with $\|v\|=\|y-z\|.$ Then, using that $\varphi$ is convex, we conclude
\begin{align*}
\|G(y)-G(z)\|_* & \leq \frac{2 M  \varphi \left( \tfrac{3}{2} \|y-z\| \right)}{\|y-z\|} =   \frac{2 M \varphi \left( \tfrac{1}{2}\|y-z\| + \|y-z\| \right)}{\|y-z\|} \leq  M\frac{\varphi(\|y-z\|) + \varphi(2\|y-z \|)}{\|y-z\|} \\
& \leq  M\omega \left( \tfrac{\|y-z\|}{2} \right) + 2 M \omega(\|y-z\|)  \leq 3 M  \omega(\|y-z\|).
\end{align*}

Let us now assume that $X$ is a Hilbert space. Note that the function $[0, \infty) \ni t \mapsto \varphi( \sqrt{t} )$ is concave because its derivative is $(0, \infty) \ni t \mapsto \omega(\sqrt{t}) / (2\sqrt{t}),$ which is non-increasing because so is $s \mapsto \omega(s) / s.$ Using the concavity of $\varphi( \sqrt{\cdot} )$ in \eqref{firstestimateMomegaG} we obtain
\begin{align*}
\langle G(z)-G(y), v \rangle & \leq M \varphi  \left( \left | \tfrac{1}{2} (z-y) + v \right | \right) + M \varphi \left( \left | \tfrac{1}{2} (y-z) + v \right | \right) \\
& \leq 2 M \varphi \left( \sqrt{\tfrac{1}{2}\left( \left | \tfrac{1}{2} (z-y) + v \right |^2 + \left | \tfrac{1}{2} (y-z) + v \right |^2 \right) }\right)  = 2 M\varphi \left( \sqrt{|v|^2 + \left | \tfrac{1}{2} (z-y) \right |^2} \right).
\end{align*}
Writing $|v|=t |y-z|$ with $t>0$ and using Proposition \ref{relationvarphivarphistar} we deduce
\begin{equation} \label{secondestimateMomegaGHilbert}
\left \langle G(z)-G(y), \frac{v}{|v|} \right \rangle  \leq \frac{2 M}{t|y-z|} \varphi \left( \sqrt{t^2 + \tfrac{1}{4}} \, |y-z|\right) \leq \frac{2M}{t} \sqrt{t^2 + \tfrac{1}{4}} \: \omega \left( \frac{\sqrt{t^2 + \tfrac{1}{4}} \, |y-z|}{2} \right).
\end{equation}

Taking $t=\sqrt{15}/2$ and $t=\sqrt{3}/2$ in \eqref{secondestimateMomegaGHilbert}, the desired estimate follows immediately.

Finally, assume that $X$ is a Hilbert space and $\omega(t)=t^\alpha$ for $\alpha \in (0,1].$ From \eqref{secondestimateMomegaGHilbert} we derive
$$
\left \langle G(z)-G(y), \frac{v}{|v|} \right \rangle  \leq \frac{2M}{t|y-z|} \varphi \left( \sqrt{t^2 + \tfrac{1}{4}} \, |y-z|\right)= \frac{2M}{(1+\alpha) \, t } \left( t^2 + \frac{1}{4} \right)^{\frac{1+\alpha}{2}} |y-z|^\alpha;
$$
where $|v|=t |y-z|$ for $t>0.$ It is straightforward to see that the function $0<t \mapsto h(t)= \frac{2(t^2 + 1/4)^{\frac{1+\alpha}{2}}}{(1+ \alpha) t}$ attains a global minimum at $t_0= 1/(2 \sqrt{\alpha})$ and $h(t_0)=\frac{2^{1-\alpha}}{\sqrt{1+\alpha}} \left( 1+\frac{1}{\alpha}\right)^{\alpha/2}.$ The desired estimate easily follows.

\medskip

$(4)$ Given $y,z\in E,$ we have
\begin{align*}
f(z) \leq f(y) + \tfrac{1}{2} \langle G(y)+G(z), z-y \rangle + M \varphi(\|y-z\|) - 2M \varphi^* \left( \tfrac{\|G(y)-G(z)\|_*}{2M} \right)\\
f(y) \leq f(z) + \tfrac{1}{2} \langle G(z)+G(y), y-z \rangle + M \varphi(\|z-y\|) - 2M \varphi^* \left( \tfrac{\|G(z)-G(y)\|_*¨}{2M} \right).
\end{align*}
By summing both inequalities we easily get 
$$
\varphi^*\left( \tfrac{1}{2M} \| G(y)-G(z)\|_{*} \right) \leq \tfrac{1}{2} \varphi( \| y-z \| ).
$$
Applying Jensen's inequality on both sides of the previous inequality (bearing in mind that $\omega^{-1}$ is convex and $\omega$ is concave) we obtain 
$$
\| G(y)-G(y) \|_{*}\, \omega^{-1}\left( \tfrac{1}{4M} \| G(y)-G(z) \|_{*} \right) \leq M \|y-z\| \omega \left( \tfrac{1}{2} \|y-z\| \right).
$$
Thus we have either $\| G(y)-G(z) \|_{*} \leq M \omega\left( \frac{\|y-z\|}{2} \right)$ or $ \omega^{-1}\left( \frac{1}{4M} \| G(y)-G(z)\|_{*} \right) \leq \| y-z\|.$ We conclude that $M_\omega(G) \leq 4M.$ 

\end{proof}

\medskip

Now we show Theorems \ref{Mainthm1}, \ref{Mainthm2}, \ref{MainthmOptimalmg}, \ref{MainthmOptimalmgForSpecialw}, \ref{Mainthm3mgSuperreflexive}, \ref{Mainthm4mgSuperreflexive}, and Corollary \ref{Various estimates corollary}. Part of the proof of these results, as those of Sections \ref{sectionboundedcase} and \ref{sectionlipschitzcase} below, will be deduced from the following technical theorem.

\begin{theorem}\label{MainTechnicalThm}
Let $E$ be a nonempty subset of a Banach space $X$, $C>0$, $(f, G):E\to\R\times X^{*}$, and $\{\psi_{y}^{+}\}_{y\in E}$, $\{\psi_{y}^{-}\}_{y\in E}$ be two families of functions $\psi_{y}^{\pm}:X\to\R$ such that
\begin{enumerate}
\item $\psi_{y}^{-}$ and $-\psi_{y}^{+}$ are strongly $C\varphi$-paraconvex for each $y\in E$;
\item $\psi_{y}^{\pm}(x)=f(y)+\langle G(y), x-y\rangle +o(|x-y|)$;
\item $\psi_{z}^{-}(x)\leq \psi_{y}^{+}(x)$ for all $x\in X$ and all $y, z\in E$.
\end{enumerate}
Let us define functions $m, g, F:X\to\R$ by
\begin{equation}\label{definition of m in terms of psiy}
m(x):=\sup_{z\in E}\psi_{z}^{-}(x),
\end{equation}
\begin{equation}\label{definition of g in terms of psiy}
g(x) :=\inf_{y\in E}\psi_{y}^{+}(x)
\end{equation}
and
\begin{equation}\label{definition of F in terms of psiy}
F(x) :=\sup\{ h(x) \, : \, h\leq g, \,\,\, h \textrm{ is strongly } C\varphi\textrm{-paraconvex}\}.
\end{equation}
Then $F$ and $-F$ are strongly $C\varphi$-paraconvex (so in particular $F\in C^{1, \omega}(X)$, with $A(F, D F)\leq C$), and $(F, D F)=(f, G)$ on $E$.

Moreover, if we also assume that $m$ is Lipschitz, with $\lip(m)\leq L$, then the function $\widetilde{F}:X\to\R$ defined by 
\begin{equation}\label{definition of Ftilde in terms of psiy}
\widetilde{F}(x) :=\sup\{ h(x) \, : \, h\leq g, \,\,\, h \textrm{ is strongly } C\varphi\textrm{-paraconvex and L-Lipschitz}\}
\end{equation}
is of class $C^{1, \omega}(X)$ and Lipschitz, with $A(\widetilde{F}, D \widetilde{F})\leq C$ and $\lip(\widetilde{F})\leq L$.
\end{theorem}
\begin{proof}

Condition $(3)$ is obviously equivalent to saying that $m$ and $g$ are finite everywhere and satisfy $m\geq g$ on $X.$ Also, if $y\in E,$ it is obvious that $m(y) \geq f(y)$ and $g(y) \leq f(y),$ which implies $m(y)=g(y)=f(y).$ Thus we have
\begin{equation}\label{properties of m g and f}
m(x)\leq  g(x) \textrm{ for all } x\in X, \textrm{ and } m(y)=g(y)=f(y) \textrm{ for all } y\in E.
\end{equation}

\begin{lemma}\label{paraconvexitygm}
{\em The functions $F$, $m$ and $-g$ are strongly $C\varphi$-paraconvex. Also, $F=f$ on $E.$}
\end{lemma}
\begin{proof}
That $m$ and $-g$ satisfy the lemma follows from the elementary observation that the supremum of a family of strongly $C\varphi$-paraconvex functions is strongly $C\varphi$-paraconvex. Once we know that $m$ is strongly $C\varphi$-paraconvex, since $m\leq g$ on $X$, we deduce that $F$ is well defined, with $m \leq F \leq g$ on $X.$ According to \eqref{properties of m g and f}, this implies $F=f$ on $E.$ Finally, applying the mentioned observation again, we obtain that $F$ is strongly $C\varphi$-paraconvex as well.
\end{proof}

\begin{lemma}
{\em The function $-F$ is strongly $C\varphi$-paraconvex.}
\end{lemma} 
\begin{proof}
Fix $x,y\in X$ and $\lambda \in [0,1]$ and define the function
$$
h(z):= \lambda F(z+(1-\lambda)(x-y))+(1-\lambda) F(z+\lambda (y-x)), \quad z\in X.
$$
Using that $F$ is strongly $C\varphi$-paraconvex it is straightforward to check that $h$ is strongly $C\varphi$-paraconvex as well. Also, since $F \leq g,$ we have that
$$
h(z) \leq \lambda g(z+(1-\lambda)(x-y))+(1-\lambda) g(z+\lambda (y-x)) \leq g(z)+ \lambda (1-\lambda) C\varphi(\|x-y\|) \quad \text{for all} \quad z\in X;
$$
where the last inequality follows from the fact that $-g$ is strongly $C\varphi$-paraconvex; see Lemma \ref{paraconvexitygm}. We have thus shown that $h-\lambda (1-\lambda) C\varphi(\|x-y \|)$ is strongly $C\varphi$-paraconvex and less than or equal to $g.$ By the definition of $F,$ we must have $h- \lambda (1-\lambda) C\varphi(\|x-y \|) \leq F.$ In particular,
$$
F(\lambda x +(1-\lambda)y) \geq h(\lambda x+(1-\lambda)y) - \lambda (1-\lambda) C \varphi(\| x-y \|)  = \lambda F(x)+(1-\lambda) F(y)- \lambda (1-\lambda) C\varphi(\|x-y \|).
$$
This proves the lemma. 
\end{proof}

\begin{lemma}
{\em We have $F\in C^{1,\omega}(X)$, with $D F=G$ on $E$, and 
$$
\| D F(x)-D F(y)\|_* \leq 3 C \omega(\|x-y \|)  \quad \text{for all} \quad x,y\in X.
$$}
\end{lemma} 
\begin{proof}
We already know that both $F$ and $-F$ are strongly $C\varphi$-paraconvex. Then by Proposition \ref{if u and -u are strongly paraconvex then u is unoformly differentiable} we have that $F$ is of class $C^{1, \omega}(X)$, with
\begin{align*}
\| D F(x)-D F(y)\|_* \leq 3 C \omega(\|x-y \|)
\end{align*}
for all $x,y \in X$, and also that
$$
A(F, D F)\leq C.
$$

Finally, let us check that $D F=G$ on $E$. By the definitions of $m$ and $g$ and the fact that $m\leq F\leq g$ on $X$  we have, for every $y\in E$ and $x\in X$,
$$
\psi_{y}^{-}(x)\leq m(x)\leq F(x)\leq g(x)\leq \psi_{y}^{+}(x),
$$
that is, $\psi_{y}^{-}\leq F\leq \psi_{y}^{+}$ on $X$, and since by condition $(2)$ we also know that $\psi_{y}^{\pm}$ is differentiable at $y$, with $D\psi_{y}^{\pm}(y)=G(y)$ and $\psi_{y}^{\pm}(y)=f(y)=F(y)$, we conclude that $DF(y)=G(y)$.
\end{proof}
The proof of Theorem \ref{MainTechnicalThm} is complete. 
\end{proof}

\medskip

\begin{proof}[{\bf Proofs of Theorems \ref{Mainthm1}, \ref{Mainthm2} and \ref{MainthmOptimalmg}.}]

Let us first note that in the case that $A(f, G)=0$ our results are trivial. Indeed, if $A(f, G)=0$, we may fix a point $z_0\in E$ and we have $f(y)+\langle G(y), x-y\rangle=f(z_0)+\langle G(z_0), x-z_0\rangle$ for all $y\in E$, $x\in X$; then the affine function $F(x):=f(z_0)+\langle G(z_0), x-z_0\rangle,  \,x\in X$, has the property that $(F, \nabla F)=(f, G)$ on $E$.  On the other hand, it is clear that $A(f, G)$ is the infimum of all constants $M>0$ for which $(mg^{1,\omega})$ holds. In particular, $A(f, G)=0$ if and only if $(f, G)$ satisfies condition $(mg^{1, \omega})$ with all $M>0$.

According to these observations, in our proofs we may assume that $A(f,G)>0$. Also note that if $A(f,G)$ is finite and strictly positive then $(f,G)$ satisfies condition $(mg^{1, \omega})$ with $M=A(f,G)$.

Let us start with the proof of Theorem \ref{MainthmOptimalmg}. To prove the necessity of $(mg^{1, \omega})$, which is obviously equivalent to $A(f, G)<\infty$, we just use Taylor's theorem: assuming $F\in C^{1, \omega}(X)$, $(F, \nabla F)=(f,G)$ on $E$, and $|\nabla F(u)-\nabla F(v)|\leq M \omega(|u-v|)$ for all $u,v\in X$, we have
\begin{eqnarray*}
& & F(z)+\langle\nabla F(z), x-z\rangle -F(y)-\langle\nabla F(y), x-y\rangle= \\
& & F(z)+\langle\nabla F(z), x-z\rangle-F(x)+F(x) -F(y)-\langle\nabla F(y), x-y\rangle \\
& & \leq M\varphi\left(|x-z|\right)+M\varphi\left(|x-y|\right),
\end{eqnarray*}
for all $x,y,z\in X,$ from which $(mg^{1, \omega})$ follows immediately (in fact this shows that $A(f, G)\leq M$).

Let us now show the sufficiency of condition $(mg^{1, \omega})$. Assume that $(f,G)$ satisfies $(mg^{1,\omega})$ with constant $M:=A(f, G)>0$. For all $y,z\in E,$ define the functions
\begin{align}\label{definitionsmandg}
\psi_{z}^{-}(x) := f(z)+\langle G(z),x-z \rangle - M \varphi(|x-z|), \quad x\in X \\
\psi_{y}^{+}(x) := f(y)+\langle G(y),x-y \rangle + M \varphi(|x-y|), \quad x\in X.
\end{align}
Condition $(mg^{1, \omega})$ tells us that $\psi_{z}^{-}(x)\leq \psi_{y}^{+}(x)$ for all $x\in X$, $y, z\in E$, so the functions $\psi_{y}^{\pm}$ meet condition $(3)$ of Theorem \ref{MainTechnicalThm}, and it is obvious from the definition that they also satisfy condition $(2)$. By Lemma \ref{constantnormomega} we have that $x\mapsto-M\varphi(|x-z|)$ is strongly $2M\varphi$-paraconvex, which immediately implies that the function
$x\mapsto f(z)+\langle G(z),x-z \rangle - M \varphi(|x-z|)=\psi_{z}^{-}(x)$ is strongly $2M\varphi$-paraconvex. Similarly we have that $x\mapsto -f(y)-\langle G(y),x-y \rangle - M \varphi(|x-y|)=-\psi_{y}^{+}(x)$ is strongly $2M\varphi$-paraconvex. Thus the families $\{\psi_{y}^{\pm}\}_{y\in E}$ satisfies all the assumptions of Theorem \ref{MainTechnicalThm}. It follows that 
$$
F(x):= \sup\lbrace h(x) \: : \: h \leq g \: \textrm{ and } \: h \: \textrm{ is } 2M \varphi\text{-strongly paraconvex} \rbrace, \quad x\in X.
$$
is of class $C^{1, \omega}(X)$, where
$$
g(x)= \inf_{y\in E} \lbrace f(y)+\langle G(y),x-y \rangle + M \varphi(|x-y|) \rbrace, \quad x\in X,
$$
and we also have $A(F, \nabla F)\leq 2 A(f, G),$ and $(F,\nabla F)=(f,G)$ on $E$. 
\end{proof}

Theorem \ref{Mainthm2} is an immediate consequence of Theorem \ref{MainthmOptimalmg} and Proposition \ref{equivalence (mg)(W)}$(3)$. Finally, in order to prove Theorem \ref{Mainthm1} we slightly modify the proof of Theorem \ref{MainthmOptimalmg}. Assume that the jet $(f,G)$ satisfies condition ($W^{1,\omega})$ with constant $M>0$ on a subset $E$ of a Hilbert space $X.$ Defining $\widetilde{\varphi} = \varphi( 2 \cdot )$ and $\psi_{y}^{\pm}(x) := f(y)+\langle G(y),x-y \rangle \pm M \widetilde{\varphi}(|x-y|)$ for every $x\in X, \, y \in E,$ we see from the arguments in the proof of Theorem \ref{MainthmOptimalmg} together with Proposition \ref{equivalence (mg)(W)}$(1)$ that the families of functions $\lbrace \psi_{y}^{\pm} \rbrace_{y\in E}$ satisfy all the assumptions of Theorem \ref{MainTechnicalThm} for $C=2M$ and with $\widetilde{\varphi}$ in place of $\varphi.$ Thus if $F$ is the function defined in \eqref{definition of F in terms of psiy} (with $\widetilde{\varphi}$ in place of $\varphi$), then both $F$ and $-F$ are strongly $2 M \widetilde{\varphi}$-paraconvex with $F=f $ and $\nabla F= G$ on $E.$ Proposition \ref{if u and -u are strongly paraconvex then u is unoformly differentiable} tells us that $F \in C^{1,\widetilde{\omega}}(X)$ with $\widetilde{\omega}= 2 \omega( 2 \cdot)$ and 
$$
|\nabla F(x)-\nabla F(y)| \leq (4 / \sqrt{3})(2 M) \, \widetilde{\omega} \left( \frac{|x-y|}{2} \right) = (16/\sqrt{3}) M \omega(|x-y|), \quad x,y\in X.
$$

\begin{proof}[{\bf Proofs of Theorem \ref{MainthmOptimalmgForSpecialw} and Corollary \ref{Various estimates corollary}.}] 
In the preceding proof we may use Lemma \ref{constantnormholder} instead of Lemma \ref{constantnormomega} to obtain that $m$ and $-g$ are strongly $2^{1-\alpha}M\varphi$-paraconvex, and the rest of the proof goes through just replacing $2M$ with $2^{1-\alpha}M$ at the appropriate points, yielding Theorem \ref{MainthmOptimalmgForSpecialw}. 

On the other hand Corollary \ref{Various estimates corollary} is an obvious consequence of previous results and some remarks made in the introduction, together with the following observation. If $\alpha \in (0,1],$ and $\omega(t)=t^\alpha,$ then one can combine Lemma \ref{constantnormholder} and Proposition \ref{if u and -u are strongly paraconvex then u is unoformly differentiable} to improve the estimates of $A(F,\nabla F)$ in Theorem \ref{MainthmOptimalmg} and of the trace seminorm $\|(f,G)\|_{E, \omega}$ in \eqref{estimation of the trace seminorm in general} as follows:
$$
A(F, \nabla F)\leq 2^{1-\alpha} A(f, G) \quad \text{and} \quad A(f,G) \leq  \|(f,G)\|_{E, \omega}\leq \frac{2^{2-2\alpha}}{\sqrt{1+\alpha}} \left( 1+\frac{1}{\alpha}\right)^{\alpha/2} A(f, G).
$$
\end{proof}

\begin{proof}[{\bf Proofs of Theorem \ref{Mainthm3mgSuperreflexive} and \ref{Mainthm4mgSuperreflexive}.}] 
We start with the proof of Theorem \ref{Mainthm3mgSuperreflexive}.  Assume that $X$ is a superreflexive space with an equivalent norm $\| \cdot \|$ satisfying \eqref{modulussmoothnesssuperreflexive} for some $\alpha\in (0,1]$ and $C>0.$ Let $\omega$ be a modulus of continuity such that $t\mapsto t^\alpha / \omega(t)$ is non-decreasing. 

\begin{lemma}\label{supperreflexivesmoothnessgeneralmodulus}
{\em If $\varphi(t)=\int_{0}^{t} \omega(s)ds,$ the function $\psi_\omega = \varphi \circ \| \cdot \|$ satisfies
$$
 \lambda \psi_\omega(x) + (1-\lambda) \psi_\omega(y)- \psi_\omega( \lambda x+(1-\lambda)y ) \leq \lambda (1-\lambda) C^* \varphi( \| x-y\| )
$$  
for every $\lambda \in [0,1]$ and every $x,y\in X;$ with $C^*= 1+ \frac{3^{1+\alpha}}{1+\alpha} C.$ }
\end{lemma}
\begin{proof}
Let us first estimate $M_\omega(D \psi_\omega).$ Combining \eqref{modulussmoothnesssuperreflexive} with Proposition \ref{if u and -u are strongly paraconvex then u is unoformly differentiable}, we obtain that $\psi_\alpha:= (1/(1+\alpha)) \| \cdot \|^{1+\alpha}$ is of class $C^{1,\alpha}(X)$ with $ \| D\psi_\alpha(x)-D \psi_\alpha (y) \|_* \leq L \|x-y\|^\alpha$ for all $x,y\in X$, where $L:=\frac{2^{-2\alpha} \, 3^{1+\alpha} C}{1+\alpha}.$ It is then obvious that the norm $\| \cdot \|$ is differentiable on $X \setminus  \lbrace 0 \rbrace,$ with $D \| \cdot \| (u) = \left( \|u\|^{1+\alpha}\right)^{\frac{1}{1+\alpha}-1} D \psi_\alpha (u).$ In particular, for any $u,v\in S_X,$ the following inequality holds
$$
\left \| D \| \cdot \| (u)- D \| \cdot \| (v) \right \|_* = \left \| D \psi_\alpha (u)- D \psi_\alpha(v) \right \|_* \leq L \| u-v\|^\alpha.
$$
Now, for any $x,y \in X \setminus \lbrace 0 \rbrace,$ assume for instance that $\| x \| \geq \|y \|$ and write
\begin{align*}
& \left \| D \psi_\omega (x)  - D \psi_\omega(y) \right \|_*  = \Big \| \omega(\| x \|) D \| \cdot \| (x)- \omega(\|y\|) D \| \cdot \| (y) \Big \|_*  \\
 & \leq | \omega ( \|x\|) - \omega(\|y\|) | \| D \| \cdot \| (y) \|_* + \omega( \|x\| ) \Big \| D \| \cdot \| (x)- D \| \cdot \| (y) \Big \|_* \\
 & \leq \omega( \|x-y\| ) + \omega( \|x\| ) \left \| D \| \cdot \|  \left( \frac{x}{\|x\|} \right)- D \| \cdot \|  \left( \frac{y}{\|y\|} \right) \right \|_* \\
 & \leq \omega( \|x-y\| )+  L \omega(\|x\|)  \left \| \frac{x}{\|x\|}-\frac{y}{\|y\|}  \right \|^\alpha \leq \omega( \|x-y\| )+  2 ^\alpha L  \frac{\omega(\|x\|)}{\|x\|^\alpha} \left \| x-y \right \|^\alpha.
\end{align*}
Using the inequality $\| x \| \geq \frac{1}{2}\| x\| + \frac{1}{2}\|y \| \geq \frac{1}{2} \|x-y\|$ and the fact that $t \mapsto t^\alpha/\omega(t)$ is non-decreasing we obtain 
$$
\left \| D \psi_\omega (x)  - D \psi_\omega(y) \right \|_* \leq \omega( \|x-y\| )+  2 ^\alpha L \frac{\omega(\|x\|)}{\|x\|^\alpha} \frac{(2\|x\|)^\alpha}{\omega(2\|x\|)} \omega( \|x-y\|) \leq ( 1 + 4^\alpha L) \omega(\|x-y\|).
$$
And, if $y=0,$ we have the inequality $\| D \psi_\omega (x)  - D \psi_\omega(y)  \|_* = \| D \psi_\omega (x) \|_* \, \omega( \| x \| ) = \omega( \|x-y\|).$ We have thus shown that $M_\omega(D \psi_\omega) \leq  1 + 4^\alpha L.$ 

Finally, in order to prove the desired inequality, let $\lambda \in [0,1]$ and $x,y\in X.$ We can easily write
\begin{align*}
&\lambda \psi_\omega(x)+(1-\lambda)\psi_\omega(y)-\psi_\omega(\lambda x+(1-\lambda)y)  \\
& = \lambda\int_0^1 \langle D \psi_\omega( \lambda x+(1-\lambda) y + t(1-\lambda)(x-y)), (1-\lambda)(x-y) \rangle dt  \\
 & \qquad \qquad + (1-\lambda) \int_0^1 \langle D \psi_\omega( \lambda x+(1-\lambda) y + t \lambda (y-x)), \lambda (y-x) \rangle dt  \\
 & = \lambda(1-\lambda) \int_0^1  \langle D \psi_\omega( \lambda x+(1-\lambda) y + t(1-\lambda)(x-y)) -  D \psi_\omega( \lambda x+(1-\lambda) y + t \lambda (y-x)), x-y \rangle dt \\
 & \leq \lambda(1-\lambda) \int_0^1 (1 + 4^\alpha L) \omega\left( t \|x-y \| \right) \|x-y \| dt = \lambda(1-\lambda) (1 + 4^\alpha L)\varphi\left( \|x-y \| \right).
\end{align*}
\end{proof}

Let us define $m$ and $g$ by
\begin{align}\label{thirddefinitionsmandg}
g(x) := \inf_{y\in E} \lbrace f(y)+\langle G(y),x-y \rangle + M \psi(x-y) \rbrace, \quad x\in X \\
m(x) := \sup_{z\in E} \lbrace f(z)+\langle G(z),x-z \rangle - M \psi(x-z) \rbrace, \quad x\in X;
\end{align}
where now $\varphi(t):=\int_{0}^{t}\omega$, and $\psi(x):=\varphi(\|x\|)$. Bearing in mind Lemma \ref{supperreflexivesmoothnessgeneralmodulus} we see that the function
$$
p(u):=M\psi(u-z)=M\varphi(\|u-z\|)
$$
satisfies
$$
\lambda p(x)+(1-\lambda)p(y)-p(\lambda x+(1-\lambda)y) \leq  \lambda (1-\lambda) C^* M \psi(x-y) \quad \text{for all} \quad \lambda \in [0,1], \: x,y\in X;
$$
where $C^*$ is as in Lemma \ref{supperreflexivesmoothnessgeneralmodulus}. That is to say, $-p$ is strongly $C^*M\varphi$-paraconvex. Then, we define
$$
F(x):= \sup\lbrace h(x) \: : \: h \leq g \: \text{ and } \: h \: \text{is strongly $C^*M\varphi$-paraconvex}  \rbrace, \quad x\in X,
$$
and exactly as in the proof of Theorem \ref{Mainthm2} one may use Theorem \ref{MainTechnicalThm} to show that $F$ and $-F$ strongly $C ^* M\varphi$-paraconvex, which by Proposition \ref{if u and -u are strongly paraconvex then u is unoformly differentiable} implies that $F$ is of class $C^{1, \omega}(X)$, with the following estimate:
$$
\|DF(x)-DF(y)\|_* \leq 3 \left(1+\tfrac{ 3^{1+\alpha}}{1+\alpha}C \right) M \omega(\|x-y\|)
$$ for every $x,y\in X$.

As in that proof, we also have $m\leq F\leq g$ on $X$, $m=f=g=F$ on $E$, and $DF=G$ on $E$.
\end{proof}

Finally, Theorem \ref{Mainthm4mgSuperreflexive} is a consequence of Theorem \ref{Mainthm3mgSuperreflexive} and Proposition \ref{equivalence (mg)(W)}. Indeed, assuming that $(f,G)$ satisfies condition $(W^{1,\omega})$ on $E,$ the functions $F$ and $-F$ are strongly $C^* M \widetilde{\varphi}$-paraconvex, with $\widetilde{\varphi}=\varphi( 2 \cdot).$ Bearing in mind that $\widetilde{\varphi}(t)= \int_0^t \widetilde{\omega};$ where $\widetilde{\omega}=2 \omega( 2 \cdot),$ the estimate in Proposition \ref{if u and -u are strongly paraconvex then u is unoformly differentiable} yields $\|DF(x)-DF(y)\|_* \leq 3 C^* M \widetilde{\omega}(\|x-y\|) \leq 12  C^*M \omega( \|x-y\|)$ for every $x,y\in X$.

\medskip

Let us conclude this section with a proof that the alternate formula \eqref{third definition of F} also provides an admissible extension $F$ in Theorem \ref{MainthmOptimalmg}. 

\begin{theorem}\label{theorem for third definition of F}
Given an arbitrary subset $E$ of a Hilbert space $X$, and a $1$-jet $(f, G):E\to\R\times X$ such that $M:=A(f, G)<\infty$, let us define $\varphi(t):=\int_{0}^{t}\omega(s)ds$, 
$$
\mathcal{F} := \left\lbrace X \ni z \mapsto a+ \langle \xi , z \rangle - \sum_{i=1}^n \lambda_i M \varphi(|z-p_i|) \, : \, a\in \R,\, \xi \in X^*, \, p_i\in X, \, \lambda_i \geq 0, \,  \sum_{i=1}^n \lambda_i=1, \, n\in \N \right\rbrace,
$$
and
$$
F(x) :=\sup\lbrace h(x) \, : \, h\in \mathcal{F}, \: h \leq g \rbrace.
$$
Then $F$ is of class $C^{1, \omega}(X)$ and satisfies $(F, \nabla F)=(f, G)$ on $E$, with
$
A(F, \nabla F)\leq 2 A(f, G).
$
\end{theorem}
\begin{proof}
By replacing $\omega(t)$ with $M\omega(t)$ if necessary, we may assume without loss of generality that $M=1$. Let us observe that:

$\bullet$ If $h\in \mathcal{F},$ then $h$ is strongly $2\varphi$-paraconvex. Indeed, let $h(z)=  a+ \langle \xi, z \rangle - \sum_{i=1}^n \lambda_i \varphi(|z-p_i|), z\in X$, where $a, \xi, \lambda_i, p_i$ are as in the definition of $\mathcal{F}.$ Then, for every $x,y\in X$ and $t \in [0,1],$ we can write
\begin{align*}
h(tx+(1-t)y)-th(x)-(1-t) h(y) & = \sum_{i=1}^n \lambda_i \left[ t \varphi(|x-p_i|)+(1-t) \varphi(|y-p_i|) - \varphi( |tx+(1-t)y-p_i| ) \right] \\
& \leq \sum_{i=1}^n \lambda_i t (1-t) 2 \varphi(|x-y|) = t(1-t) 2 \varphi(|x-y|);
\end{align*}
where we have used that $-\varphi \circ | \cdot |$ is strongly $2\varphi$-paraconvex.

\medskip

$\bullet$ $F$ is well defined and satisfies $F \leq g$. Indeed, since $m \leq g$ and any function $x \mapsto f(y)+\langle G(y), x-y\rangle -\varphi(|x-y|)$ belongs to $\mathcal{F},$ we have that $F$ is well defined and $m \leq F \leq g$ on $X.$ We then deduce that $F=f$ on $E$ and 
$$
f(y)+ \langle G(y), x-y \rangle-\varphi(|x-y|) \leq F(x) \leq f(y)+\langle G(y), x-y \rangle + \varphi(|x-y|), \quad x\in X, \, y \in E.
$$
This shows that $F$ is differentiable on $E$, with $\nabla F=G$ on $E.$

\medskip

$\bullet$ $F$ is strongly $2\varphi$-paraconvex. This is a consequence of the general and obvious fact that the supremum of a family of strongly $C \varphi$-paraconvex functions is also strongly $C \varphi$-paraconvex.   

\medskip

In order to show that $F\in C^{1, \omega}(X)$, let us also note the following.

$\bullet$ The function $-F$ is strongly $2\varphi$-paraconvex. Indeed, let $x,y\in X,\, t\in [0,1]$ and $\varepsilon>0.$ We can find $h_1, h_2 \in \mathcal{F}$ with $h_i \leq g$ and $F(x) \leq h_1(x) + \varepsilon, \, F(y) \leq h_2(y)+ \varepsilon.$ Define 
$$
h(z)=t h_1( z + (1-t)(x-y) ) + (1-t) h_2(z+t(y-x)), \quad z\in X.
$$
It is straightforward to see that $h \in \mathcal{F}.$ We have that
$$
h(z) \leq t g(z+(1-t)(x-y))+(1-t) g(z+t (y-x)) \leq g(z)+ \lambda (1-t) 2\varphi(|x-y|) \quad \text{for all} \quad z\in X;
$$
where we have used the fact that $-g$ is strongly $2\varphi$-paraconvex. This shows that $h-t(1-t)2\varphi(|x-y|) \leq g.$ On the other hand, we have that $h-t(1-t)2\varphi(|x-y|)\in \mathcal{F}.$ This implies, by the definition of $F,$ that $h-t(1-t)2\varphi(|x-y|) \leq F.$ And we can write
\begin{eqnarray*}
& & tF(x)+(1-t)F(y)-F(tx+(1-t)y) \\
& & \leq t h_1(x) + (1-t) h_2(y) +\varepsilon-h(tx+(1-t)y)+t(1-t)2\varphi(|x-y|) \\
& & = t(1-t)2\varphi(|x-y|) + \varepsilon.
\end{eqnarray*}
Letting $\varepsilon$ go to $0$ we thus obtain that $-F$ is strongly $2\varphi$-paraconvex. 

Now we can apply Proposition \ref{if u and -u are strongly paraconvex then u is unoformly differentiable} to conclude that $F\in C^{1,\omega}(X)$ and $A(F, \nabla F)\leq 2= 2A(f, G)$.
\end{proof}

\medskip

\section{The Bounded Case}\label{sectionboundedcase}

If a jet $(f, G)$ defined on a subset $X$ of a Hilbert space $X$ satisfies $A(f, G)<\infty$ then we already know that there exists $F\in C^{1, \omega}(X)$ such that $(F, \nabla F)=(f, G)$ on $E$. If the given functions $f, G$ are bounded on $E$, then it is natural to ask whether $(F, \nabla F)$ can be taken to be bounded. The extensions $F$ defined by \eqref{first definition of F} may not be bounded (in fact they are {\em never} bounded when $E$ is bounded), but in this section we will see how we can modify the proof of Theorems \ref{Mainthm1} and \ref{Mainthm2} so as to get $(F, \nabla F)$ bounded.
Also, with a different modification of the proof, we can obtain a certain continuous dependence of the extensions on the initial data, meaning that if a sequence $\{(f_n, G_n)\}_{n\in\N}$ of jets converges uniformly on $E$ to a jet $(f, G)$ then the corresponding extensions satisfy that $\lim_{n\to\infty}(F_n, \nabla F_n)= (F, \nabla F)$ uniformly on $X$.

In order to formulate our results more precisely, let us introduce some more notation.
Let us denote
$$
C^{1,\omega}_{b}(X)=\{h\in C^{1, \omega}(X)  \, : \, (h, \nabla h) \textrm{ is bounded on } X\}, 
$$
and endow this vector space with the norm
$$
\|h\|_{1, \omega, b}:=\|h\|_{\infty} +\|\nabla h\|_{\infty} +M_{\omega}(\nabla h),
$$
which makes $C^{1, \omega}_{b}(X)$ a Banach space. Also observe that the mapping $(f, G)\mapsto A(f, G)$ is a seminorm on the vector space of $1$-jets
$$
\mathcal{J}^{1, \omega}_{b}(E) :=\{ (f, G):E\to\R\times X \, : \, A(f, G)<\infty \text{ and } (f, G) \textrm{ is bounded on } E\},
$$
and therefore
$$
\rho(f,G) :=\|f\|_{\infty} +\|G\|_{\infty}+A(f,G)
$$
defines a norm on $\mathcal{J}^{1, \omega}_{b}(E)$.

\begin{theorem}\label{MainThmBoundedCase}
Let $X$ be a Hilbert space. There exist a nonlinear operator  $\mathcal{E}:\mathcal{J}^{1, \omega}_{b}(E)\to C^{1, \omega}_{b}(X)$ and a constant $C>0$, only depending on $\omega$, with the following properties:
\begin{enumerate}
\item $\left(\mathcal{E}(f,G), \nabla\mathcal{E}(f,G)\right)_{|_E}=(f, G)$ for every $(f, G)\in \mathcal{J}^{1, \omega}_{b}(E)$; 
\item $\|\mathcal{E}(f, G)\|_{1, \omega, b}\leq C\,\rho(f, G)$ for all $(f, G)\in \mathcal{J}^{1, \omega}_{b}(E)$.
\end{enumerate}
\end{theorem}
\begin{proof}
Given a jet $(f, G)\in \mathcal{J}^{1, \omega}_{b}(E)$, let us define $\mathcal{E}(f,G)$ as follows. For the number
$$
M :=\max\left\{ \frac{3}{\varphi(1)}\left( \|f\|_{\infty}+ \|G\|_{\infty}\right), \, A(f, G)\right\},
$$
let us set
\begin{align*}
m(x) :=\max\left\{ -2(\|f\|_\infty+ \|G\|_{\infty}), \, \sup_{z\in E}\{f(z)+\langle G(z), x-z\rangle -M\varphi(|x-z|) \} \right\}, \\
g(x) :=\min\left\{ 2(\|f\|_{\infty}+\|G\|_{\infty}), \, \inf_{y\in E}\{f(y)+\langle G(y), x-y\rangle +M\varphi(|x-y|) \}\right\}, \\
F(x) :=\sup\{h(x) \, : \, h\leq g \textrm{ on } X, \textrm{ and } h \textrm{ is strongly } 2M\varphi\textrm{-paraconvex}\},
\end{align*}
and 
$$
\mathcal{E}(f,G) :=F.
$$
Let us check that the mapping $(f,G)\mapsto \mathcal{E}(f,G)$ satisfies the properties of the statement.

If $|x-y|\geq 1$ we have (bearing in mind that $\varphi(t)\geq \varphi(1)t$ for $t\geq 1$ since $\varphi$ is convex)
\begin{equation*}
f(y)+\langle G(y), x-y\rangle -M\varphi(|x-y|)\leq \|f\|_{\infty} +\left( \|G\|_{\infty}-M\varphi(1)\right) |x-y|\leq \|f\|_{\infty}.
\end{equation*}
On the other hand, if $|x-y|<1$ then
$$
f(y)+\langle G(y), x-y\rangle -M\varphi(|x-y|)\leq \|f\|_{\infty}+\|G\|_{\infty}.
$$
In either case we have
\begin{equation}\label{controlling the parabolas 1}
f(y)+\langle G(y), x-y\rangle -M\varphi(|x-y|)\leq \|f\|_{\infty}+\|G\|_{\infty}
\end{equation}
for all $x\in X$, $y\in E$, and similarly we see that
\begin{equation}\label{controlling the parabolas 2}
f(z)+\langle G(z), x-z\rangle+M\varphi(|x-z|)\geq -\|f\|_{\infty}-\|G\|_{\infty}
\end{equation}
for all $x\in X$, $z\in E$.

Let us define, for each $y\in E$, the functions
\begin{align}\label{definition of psiy bounded case}
\psi_{y}^{-}(x) :=\max\left\{ -2(\|f\|_{\infty}+\|G\|_{\infty}), \, f(y)+\langle G(y), x-y\rangle -M\varphi(|x-y|) \right\}, \\
\psi_{y}^{+}(x) :=\min\left\{ 2(\|f\|_{\infty}+\|G\|_{\infty}), \, f(y)+\langle G(y), x-y\rangle +M\varphi(|x-y|) \right\}.
\end{align}
By using \eqref{controlling the parabolas 1} and \eqref{controlling the parabolas 2} and the assumption that $A(f,G)\leq M<\infty$, it is immediately checked that these functions satisfy conditions $(2)$ and $(3)$ of Theorem \ref{MainTechnicalThm}. Besides, recalling Lemma \ref{constantnormomega} and the fact that the maximum of two strongly $2M\varphi$-paraconvex functions is strongly $2M\varphi$-paraconvex, we also have that $m$ is strongly $2M\varphi$-paraconvex. Similarly, $-g$ is strongly $2M\varphi$-paraconvex too. Then we can apply Theorem \ref{MainTechnicalThm} with $C=2M$, obtaining that $F$ and $-F$ are strongly $2M\varphi$-paraconvex, hence $F\in C^{1, \omega}(X)$, and that $(F, \nabla F)=(f, G)$ on $E$, with
\begin{equation}\label{estimate for MwnablaF bounded case}
M_{\omega}(\nabla F)\leq 3 A(F, \nabla F)\leq  6 M.
\end{equation}
Since $m\leq F\leq g$, it is obvious that we also have
\begin{equation}\label{estimate for F bounded case}
\|F\|_{\infty}\leq 2 ( \|f\|_{\infty}+\|G\|_{\infty}).
\end{equation}
Let us now estimate $\|\nabla F\|_{\infty}$. By using \eqref{inequality for differentiability of u} with $s=1$ in the proof of Proposition \ref{if u and -u are strongly paraconvex then u is unoformly differentiable}, and recalling that both $F$ and $-F$ are strongly $2M\varphi$-paraconvex, we have
$$
\left| F(a+h)-F(a)- \frac{F(a+th)-F(a)}{t} \right|\leq 2M(1-t)\varphi(|h|),
$$
hence
$$
\left| \frac{F(a+th)-F(a)}{t} \right|\leq 2\|F\|_{\infty}+ 2M(1-t)\varphi(|h|),
$$
and by setting $|h|=1$ and letting $t\to 0^{+}$ we obtain
\begin{equation}\label{estimate for Lip F bounded case}
\|\nabla F\|_{\infty}\leq 2\|F\|_{\infty}+2M\varphi(1).
\end{equation}
In conclusion, by combining \eqref{estimate for MwnablaF bounded case}, \eqref{estimate for F bounded case} and \eqref{estimate for Lip F bounded case} we obtain that
$$
\|F\|_{1, \omega, b}=\|F\|_{\infty}+\|\nabla F\|_{\infty}+M_{\omega}( \nabla F)\leq C \left(\|f\|_{\infty}+\|G\|_{\infty}+A(f,G)\right) = C\, \rho(f, G),
$$
where $C>0$ is a constant only depending on $\omega$. 
\end{proof}

\begin{theorem}\label{MainThmBoundedCaseWithContinuity}
Let $X$ be a Hilbert space. There exists a constant $C>0$, only depending on $\omega$, such that for each number $A>0$ there exists a nonlinear operator $$\mathcal{E}_{A}:\mathcal{J}^{1, \omega}_{b}(E)\cap\{(f,G) : A(f,G)\leq A\}\to C^{1, \omega}_{b}(X)$$ with the following properties:
\begin{enumerate}
\item $\left(\mathcal{E}_A(f,G), \nabla\mathcal{E}_A(f,G)\right)_{|_E}=(f,G)$ for all $(f, G)\in \mathcal{J}^{1, \omega}_{b}(E)\cap\{(f,G) : A(f,G)\leq A\}$; 
\item $\|\mathcal{E}_A(f, G)\|_{1, \omega, b}\leq C \left( \|f\|_\infty+\|G\|_\infty+ A \right)$ for all $(f, G)\in \mathcal{J}^{1, \omega}_{b}(E)\cap\{(f,G) : A(f,G)\leq A\};$ 
\item If $\{(f_n, G_n)\}_{n\in\N}\subset \mathcal{J}^{1, \omega}_{b}(E)$ and $(f,G)\in \mathcal{J}^{1, \omega}_{b}(E)$ are such that $A(f_n,G_n)\leq A$, $A(f,G) \leq A$, and $(f_n, G_n)$ converges to $(f, G)$ uniformly on $E$, then $(\mathcal{E}_{A} (f_n, G_n), \nabla\mathcal{E}_{A} (f_n, G_n))$ converges to $(\mathcal{E}_{A}(f,G), \nabla \mathcal{E}_{A} (f,G))$ uniformly on $X$.
\end{enumerate}
\end{theorem}
\begin{proof}
Given a jet $(f, G)\in \mathcal{J}^{1, \omega}_{b}(E)\cap\{(f,G) : A(f,G)\leq A\}$, let us define $\mathcal{E}(f,G)$ as follows. For the number
$$
M := 3 \varphi(1)^{-1} \left( \|f\|_\infty + \|G\|_\infty \right) + A,
$$
let us set
\begin{align*}
g(x) :=\min\left\{ 2(\|f\|_\infty+\|G\|_{\infty}), \, \inf_{y\in E}\{f(y)+\langle G(y), x-y\rangle +M\varphi(|x-y|) \} \right\}, \\
F(x) :=\sup\{h(x) \, : \, h\leq g \textrm{ on } X, \textrm{ and } h \textrm{ is strongly } 2M\varphi\textrm{-paraconvex}\},
\end{align*}
and 
$$
\mathcal{E}_{A}(f,G) :=F.
$$
That the mapping $(f,G)\mapsto \mathcal{E}(f,G)$ satisfies properties $(1)$ and $(2)$ of the statement can be checked exactly as in the proof of the previous theorem. 

In order to prove $(3)$ we need to localize the infimum defining the associated functions $g$ and
$$
g_n(x):=\min\left\{ 2(\|f_n\|_{\infty}+\|G_n\|_{\infty}), \, \inf_{y\in E}\{f_n(y)+\langle G_n(y), x-y\rangle +M_n\varphi(|x-y|) \} \right\},
$$
where $M_n :=3 \varphi(1)^{-1} \left( \|f_n\|_\infty + \|G_n\|_\infty \right) + A$ for every $n\in \N.$ 
\begin{lemma}{\em
We have that 
$$
g_n(x)=\inf\left( \left\{ 2(\|f_n\|_{\infty}+\|G_n\|_{\infty})\right\}\bigcup  \left\{f_n(y)+\langle G_n(y), x-y\rangle +M_n \varphi(|x-y|) \, : \, y\in E\cap B(x, 1)\right\}\right)
$$
and
\begin{equation}\label{localization of g}
g(x)=\inf\left( \left\{ 2(\|f\|_{\infty}+\|G\|_{\infty})\right\}\bigcup  \left\{f(y)+\langle G(y), x-y\rangle +M\varphi(|x-y|) \, : \, y\in E\cap B(x, 1)\right\}\right)
\end{equation}
for all $x\in X$, $n\in\N$.}
\end{lemma}
\begin{proof}
If $x\in X$, $y\in E$ and $|x-y|\geq 1$ then
\begin{eqnarray*}
& & f(y)+\langle G(y), x-y\rangle +M\varphi(|x-y|)\geq -\|f\|_{\infty}-\|G\|_{\infty}|x-y|+M\varphi(1)|x-y|\\
& & \geq \left( M\varphi(1)-\|G\|_{\infty}\right) |x-y|-\|f\|_{\infty}\geq
3(\|f\|_{\infty}+\|G\|_{\infty})-\|G\|_{\infty}-\|f\|_{\infty} \\
& & =2(\|f\|_{\infty}+\|G\|_{\infty}).
\end{eqnarray*}
Therefore
$$
g(x)=\inf\left( \left\{ 2(\|f\|_\infty+\|G\|_{\infty})\right\}\bigcup  \left\{f(y)+\langle G(y), x-y\rangle +M\varphi(|x-y|) \, : \, y\in E\cap B(x, 1)\right\}\right).
$$
Obviously the same holds true of $g_n$.
\end{proof}

\begin{lemma}\label{gn converges to g uniformly}
{\em $(g_n)$ converges to $g$ uniformly on $X$.}
\end{lemma}
\begin{proof}
Let $\varepsilon>0$, and choose $n_0\in\N$ large enough so that
$$
2\left(\|f_n-f\|_{\infty}+\|G_n-G \|_{\infty}\right)\leq \varepsilon \quad \text{and} \quad \varphi(1) |M_n-M| \leq \varepsilon/4 \quad \textrm{ for all } n\geq n_0.
$$
Then, given $x\in X$, we either have $g(x)= 2(\|f\|_{\infty}+\|G\|_{\infty})$ or $g(x)<2(\|f\|_{\infty}+\|G\|_{\infty})$. In the first case we have
$$
g_n(x)-g(x)\leq 2(\|f_n\|_{\infty}+\|G_n\|_{\infty})-2(\|f\|_{\infty}+\|G\|_{\infty})\leq
2(\|f_n-f\|_{\infty}+\|G_n-G\|_{\infty})\leq \varepsilon
$$
for all $n\geq n_0$. In the second case, thanks to \eqref{localization of g} we may find $y_x\in E\cap B(x, 1)$ such that
$$
g(x)+\varepsilon/4 > f(y_x)+\langle G(y_x), x-y_x\rangle +M\varphi(|x-y_x|),
$$
hence, for all $n\geq n_0$,
\begin{align*}
g_n(x)-g(x)& \leq f_{n}(y_x)+\langle G_{n}(y_x), x-y_x\rangle +M_n \varphi(|x-y_x|)
-f(y_x)-\langle G(y_x), x-y_x\rangle -M \varphi(|x-y_x|) + \varepsilon/4\\
& =f_{n}(y_x)-f(y_x)+\langle G_{n}(y_x)-G(y_x), x-y_x\rangle + |M_n-M| \varphi(|x-y_x|)+ \varepsilon/4 \\
& \leq \|f_n-f\|_{\infty}+\|G_n-G\|_{\infty} + \varphi(1) |M_n-M| + \varepsilon/4 \leq\varepsilon.
\end{align*}
In either case we see that if $n\geq n_0$ then
$$
g_n(x)-g(x)\leq\varepsilon
$$
for all $x\in X$. Similarly one can check that
$$
g(x)-g_n(x)\leq\varepsilon
$$
for all $x\in X$, $n\geq n_0$. Thus we conclude that $\|g_n-g\|_{\infty}\leq \varepsilon$ for all $n\geq n_0$.
\end{proof}

\begin{lemma}\label{Fn converges to F uniformly}
{\em $(F_n)$ converges to $F$ uniformly on $X$.}
\end{lemma}
\begin{proof}
Observe that the family of functions $\lbrace g, g_n, F, F_n \rbrace_n$ is uniformly bounded thanks to property $(2)$ and the fact that $\lbrace(f_n,G_n)\rbrace_n$ converges uniformly to $(f,G).$ Together with Lemma \ref{gn converges to g uniformly}, this implies that, given $\varepsilon>0$, we can choose $n_0\in\N$ so that, for every $n \geq n_0$
\begin{equation}\label{firstroundinequalitiesepsilon}
M_n M^{-1} \leq 2, \quad \max \left\lbrace \|g_n-g\|_\infty, \left |M M_n^{-1}-1 \right | \|g_n\|_\infty, \left |M_n M^{-1}-1 \right | \|F\|_\infty \right\rbrace \leq \varepsilon/6,
\end{equation}
\begin{equation}\label{secondroundinequalitiesepsilon}
M M_n^{-1} \leq 2, \quad \max \left\lbrace \|g_n-g\|_\infty, \left |M_n M^{-1}-1 \right | \|g\|_\infty, \left |M M_n^{-1}-1 \right | \|F_n\|_\infty \right\rbrace \leq \varepsilon/6.
\end{equation}
In particular, \eqref{firstroundinequalitiesepsilon} implies that
\begin{equation}\label{thirdroundinequalities}
M M_n^{-1}g_n \leq g+\left |M M_n^{-1}-1 \right| \|g_n\|_\infty + \varepsilon/6 \leq g + \varepsilon/3 \quad \text{for each} \quad n \geq n_0.
\end{equation}
Observing that a function $h$ is strongly $a\varphi$-paraconvex if and only if $b a^{-1} h + c$ is strongly $b\varphi$-paraconvex, where $a,b>0$ and $c\in \R$ are any constants, the inequalities in \eqref{firstroundinequalitiesepsilon} and \eqref{thirdroundinequalities} yield, for every $x\in X$ and $n \geq n_0,$
\begin{eqnarray*}
& & F_n(x)= \sup\{ h(x) \, : \, h\leq g_n, \,\,\, h \textrm{ is strongly } 2M_n\varphi\textrm{-paraconvex}\} \\
 & & =M_n M^{-1}\sup\{ h(x) \, : \, h\leq M M_n^{-1} g_n, \,\,\, h \textrm{ is strongly } 2M\varphi\textrm{-paraconvex}\} \\
  & & \leq M_n M^{-1}\sup\{ h(x) \, : \, h\leq g+  \varepsilon/3, \,\,\, h \textrm{ is strongly } 2M\varphi\textrm{-paraconvex}\} \\
& & = M_n M^{-1}\sup\{ h(x) \, : \, h\leq g, \,\,\, h \textrm{ is strongly } 2M\varphi\textrm{-paraconvex}\} + M_n M^{-1}\varepsilon /3  \\
& & = M_n M^{-1} F(x) + M_n M^{-1}\varepsilon /3 \leq  F(x)+ \left | M_n M^{-1} -1 \right | \|F\|_\infty + 2\varepsilon/3 \\
& & \leq F(x)+ \varepsilon/6 + 2\varepsilon/3 \leq F(x) + \varepsilon.
\end{eqnarray*}
Similarly (using the inequalities of \eqref{secondroundinequalitiesepsilon}), we obtain
$$
F(x)\leq F_n(x)+\varepsilon
$$
for all $x\in X$, $n\geq n_0$.
Therefore $\|F_n-F\|_{\infty}\leq\varepsilon$ for all $n\geq n_0$.
\end{proof}

It only remains to be shown that $\lim_{n\to\infty}\|\nabla F_n-\nabla F\|_{\infty}=0$. This is a consequence of the following fact (which is of course well known; we include a short proof here for the reader's convenience).
\begin{lemma}\label{uniform convergence of gradients}
{\em Let $u:X\to\R$ be differentiable, and let $(u_k)$ be a sequence of differentiable functions such that $u_k$ converges to $u$ uniformly on $X$, and such that for some constant $c>0$, we have
\begin{equation*}
-c\varphi\left(|x-y|\right)\leq u_k(y)-u_k(x)- \langle \nabla u_k(x), y-x\rangle\leq c\varphi\left(|x-y|\right) 
\end{equation*}
and
\begin{equation*}
-c\varphi\left(|x-y|\right)\leq u(y)-u(x)- \langle \nabla u(x), y-x\rangle\leq c\varphi\left(|x-y|\right)
\end{equation*}
for all $k \in \N$ and all $x, y\in X$. Then $\|\nabla u_k -\nabla u\|_\infty$ converges to $0$ uniformly on $X$.}
\end{lemma}
\begin{proof}
By substracting the second inequality from the first one we get
\begin{eqnarray}\label{estimate for uniform convergence of gradients}
\langle \nabla u(x)-\nabla u_k(x), y-x\rangle \leq u_k(x)-u(x)+u(y)-u_k(y)+2c\varphi(|y-x|).
\end{eqnarray}
Given $\varepsilon>0$ we may choose $k_0\in\N$ so that $\|u_k-u\|_{\infty}\leq\varepsilon^2/4$ for all $k\geq k_0$. By taking $h\in X$ with $|h|=\varepsilon$ and $y=x+h$ in \eqref{estimate for uniform convergence of gradients} we obtain
$$
\langle \nabla u(x)-\nabla u_k(x), h\rangle \leq \varepsilon^2/2 +2c\varphi(\varepsilon),
$$
hence
$$
|\nabla u(x)-\nabla u_k(x)|=\sup_{|h|=\varepsilon}\langle \nabla u(x)-\nabla u_k(x), \varepsilon^{-1} h\rangle \leq \varepsilon/2 +2c\varphi(\varepsilon)/\varepsilon
$$
for all $k\geq k_0$, $x\in X$, and since $\varphi(t)=o(t)$ this implies that $\lim_{k\to\infty}\|\nabla u_k-\nabla u\|_{\infty}=0$.
\end{proof}
It is clear that property $(2)$ together with the fact that $\lbrace(f_n,G_n)\rbrace_n$ converges uniformly to $(f,G)$ imply that $\max \lbrace M_\omega(\nabla F_n), \, M_\omega(\nabla F) \rbrace\leq A^*C,$ for $n$ large enough and for a constant $A^*>0$ comparable to $\|f\|_\infty+\|G\|_\infty+A.$ Combining Lemma \ref{Fn converges to F uniformly} and this observation, we can apply Lemma \ref{uniform convergence of gradients} for $(F_n)_n$ and $F$ to conclude $\lim_{n\to\infty}\nabla F_n=\nabla F$ uniformly on $X$. 

The proof of Theorem \ref{MainThmBoundedCaseWithContinuity} is complete.
\end{proof}

\begin{remark}
{\em \hfill{}
\item[] $(1)$ Theorems \ref{MainThmBoundedCase} and \ref{MainThmBoundedCaseWithContinuity} have analogues for superreflexive spaces $X$ and the classes $C^{1,\omega}(X),$ assuming that $t^\alpha /\omega(t)$ is a nondecreasing function. We let the reader formulate them. The proofs are the same, with obvious changes.

\item[] $(2)$ It would be interesting to know whether one can improve Theorems \ref{Mainthm2} and \ref{MainThmBoundedCase} to find an extension operator with the additional property that $$\lim_{n\to\infty}A( \mathcal{E}(f_n, G_n)-\mathcal{E}(f,G), \nabla\mathcal{E}(f_n-f, G_n)-\nabla\mathcal{E}(f,G))=0$$ whenever $A(f_n-f, G_n-G)\to 0$, or at least such that $\mathcal{E}$ be continuous from the normed space $(\mathcal{J}^{1, \omega}_{b}(E), \rho)$ into $( C^{1, \omega}_{b}(X), \|\cdot\|_{1, \omega, b})$.}
\end{remark}

\section{The Lipschitz Case}\label{sectionlipschitzcase}

In this section we will show a variant of our main result in which we are given $(f, G)$ with $G$ bounded but $f$ unbounded, and we want an extension $F$ with $\nabla F$ bounded.

\begin{theorem}\label{thm for Lipschitz}
Let $E$ be an arbitrary subset of a Hilbert space $X$, and let $(f, G):E\to\R\times X$ be a $1$-jet. Then there exists $F\in C^{1, \omega}(X)$ with $F$ Lipschitz and $(F, \nabla F)_{|_E}=(f, G)$ if and only if $f$ is Lipschitz, $G$ is bounded and $A(f,G)<\infty.$ In such case we can take $F$ with the additional properties that $A(F, \nabla F)+\lip(F)\leq C \left( A(f, G)+ \lip(f) + \|G\|_\infty \right)$, where $C$ is a constant depending only on $\omega$.
\end{theorem}
\begin{proof}
Given $\omega:[0, \infty)\to [0, \infty)$, let us define
$$
\widetilde{\varphi}(t)=
\begin{cases}
\varphi(t)=\int_{0}^{t}\omega(s)ds & \textrm{ if } t\leq 1 \\
\varphi(1)+\omega(1)(t-1) & \textrm{ if } t\geq 1.
\end{cases}
$$
Let us consider the function $\widetilde{\psi}:=\widetilde{\varphi} \circ | \cdot |,$ and check that $-\widetilde{\psi}$ is strongly $C \widetilde{\varphi}$-paraconvex for some absolute constant $C>0$. Indeed, if $\widetilde{\omega}$ is defined as $\widetilde{\omega}=\omega $ on $[0,1]$ and $\widetilde{\omega}=\omega(1)$ on $[1,+\infty),$ observe that for all $u,v\in X$ such that $|u|, |v| \geq 1,$ we have that  
$$
|\nabla \widetilde{\psi}(u)-\nabla \widetilde{\psi}(v)| = \omega(1) \bigg | \frac{u}{|u|} - \frac{v}{|v|} \bigg |  \leq \frac{2 \, \omega(1)|u-v|}{\max(|u|,|v|)} \leq 4 \, \widetilde{\omega}(|u-v|).
$$
And if $u,v\in X$ are such that $|u|, |v| \leq 1,$ combining Lemma \ref{constantnormomega} with Proposition \ref{if u and -u are strongly paraconvex then u is unoformly differentiable} we obtain $| \nabla (\varphi \circ |\cdot|)(u) - \nabla (\varphi \circ |\cdot|) (v) | \leq A \omega(|u-v|)$ for an absolute constant $A>0.$ This implies
$$
| \nabla \widetilde{\psi}(u) - \nabla \widetilde{\psi}(v) | \leq A \omega(|u-v|) \leq 2A \omega\left(\frac{|u-v|}{2} \right)  = 2A \widetilde\omega\left(\frac{|u-v|}{2} \right)  \leq 2A \widetilde\omega(|u-v|).
$$
If one of the vectors, say $u,$ is inside the unit ball and the other is not, then the line segment $[u,v]$ intersects the unit sphere at a unique point $z,$ and we have
$$
| \nabla \widetilde{\psi}(u) - \nabla \widetilde{\psi}(v) | \leq  | \nabla \widetilde{\psi}(u) - \nabla \widetilde{\psi}(z) | + | \nabla \widetilde{\psi}(z) - \nabla \widetilde{\psi}(v) | \leq 2A \, \widetilde\omega(|u-z|)+ 4  \, \widetilde{\omega}(|z-v|) \leq (2A+4) \widetilde\omega(|u-v|).
$$
In any case we see that 
\begin{equation}\label{estimationnormboundedmodulus}
| \nabla \widetilde{\psi}(u) - \nabla \widetilde{\psi}(v) | \leq (2A+ 4) \widetilde{\omega}(|u-v|) \quad \text{for all} \quad u,v\in X.
\end{equation}
Now, given $x,y\in X$ and $\lambda \in [0,1],$ we can use \eqref{estimationnormboundedmodulus} to obtain
\begin{align*}
&\lambda \widetilde{\psi}(x)+(1-\lambda)\widetilde{\psi}(y)-\widetilde{\psi}(\lambda x+(1-\lambda)y)  \\
& = \lambda\int_0^1  \big \langle \nabla \widetilde{\psi}( \lambda x+(1-\lambda) y + t(1-\lambda)(x-y)), (1-\lambda)(x-y) \big \rangle dt  \\
 & \qquad \qquad + (1-\lambda) \int_0^1 \big \langle \nabla \widetilde{\psi}( \lambda x+(1-\lambda) y + t \lambda (y-x)), \lambda (y-x) \big \rangle dt  \\
 & = \lambda(1-\lambda) \int_0^1 \big \langle \nabla \widetilde{\psi}( \lambda x+(1-\lambda) y + t(1-\lambda)(x-y)) -  \nabla \widetilde{\psi}( \lambda x+(1-\lambda) y + t \lambda (y-x)), x-y \big \rangle dt \\
 & \leq \lambda(1-\lambda) \int_0^1 (2A+4)\widetilde{\omega}\left( t |x-y| \right) |x-y| dt = \lambda(1-\lambda) (2A+4)\widetilde{\varphi}\left( |x-y| \right).
\end{align*}
This proves that $-\widetilde{\psi}$ is strongly $C \widetilde{\varphi}$-paraconvex, with $C=2A+4$. This $C$ is not to be confused with that of the statement of the theorem.

Now let $M:= A(f,G)<\infty$ and let $y,z\in E,\, x\in X.$ Observe that if $|x-y|, |x-z| \leq 1,$ then 
$$
|f(y)+\langle G(y), x-y\rangle- f(z)-\langle G(z), x-z\rangle | \leq M \left( \varphi(|x-y|) + \varphi(|x-z|) \right) = M \left( \widetilde{\varphi}(|x-y|) + \widetilde{\varphi}(|x-z|) \right).
$$
On the other hand if $|x-y|>1$ or $|x-z| >1,$ using that $f$ is Lipschitz, $G$ is bounded, the fact that $t \leq \varphi(1)^{-1} \widetilde{\varphi}(t)$ for every $t \geq 1,$ and finally the convexity of $\widetilde{\varphi}$, we can write
\begin{align*}
 |f(y)+ & \langle G(y), x-y\rangle- f(z)-\langle G(z), x-z\rangle | \leq \lip(f) |y-z| + \|G\|_\infty \left( |x-y| + |x-z| \right) \\
& \leq \left( \lip(f)+\|G\|_\infty \right) \left( |x-y| + |x-z| \right) \leq \frac{\lip(f)+ \|G\|_\infty}{\varphi(1)} \widetilde{\varphi} \left(  |x-y| + |x-z| \right) \\
& \leq \frac{\lip(f)+ \|G\|_\infty}{2 \varphi(1)}  \left( \widetilde{\varphi} (  2|x-y| ) + \widetilde{\varphi}( 2|x-z| ) \right) \leq \frac{2 \left( \lip(f)+ \|G\|_\infty \right)}{\varphi(1)}  \left( \widetilde{\varphi} (  |x-y| ) + \widetilde{\varphi}( |x-z| ) \right).
\end{align*}
We conclude that
$$
|f(y)+\langle G(y), x-y\rangle- f(z)-\langle G(z), x-z\rangle | \leq \widetilde{M} \left( \widetilde{\varphi}(|x-y|) + \widetilde{\varphi}(|x-z|) \right) \quad y,z\in E, \, x\in X;
$$
where $\widetilde{M} = \max\left( M, 2 \, \varphi(1)^{-1}\left( \lip(f) + \|G\|_\infty \right) \right).$ 

The preceding observations show that the family of functions
$$
\psi_{y}^{\pm}(x)=f(y)+\langle G(y), x-y\rangle \pm \widetilde{M}\widetilde{\varphi}(|x-y|), \quad x\in X,\: y\in E,
$$
satisfy conditions $(1), (2)$ and $(3)$ of Theorem \ref{MainTechnicalThm}. In addition, since $(\widetilde{\varphi})'= \widetilde{\omega} \leq \omega(1),$ the function $m:= \sup_{y\in E} \psi_y^-$ is Lipschitz with $\lip(m) \leq \|G\|_\infty+ \omega(1) \widetilde{M}.$ Applying Theorem \ref{MainTechnicalThm} (by means of formula \eqref{definition of Ftilde in terms of psiy}), we obtain a Lipschitz function $F \in C^{1,\widetilde{\omega}}(X)$ such that $(F,\nabla F)=(f,G)$ on $E,$ $A(F,\nabla F) \leq C \widetilde{M}$ and $\lip(F) \leq \|G\|_\infty+ \omega(1) \widetilde{M}.$ Notice that Theorem \ref{MainTechnicalThm} can be applied for $\widetilde{\omega}$ and $\widetilde{\varphi}$ since the assumption that the modulus of continuity $\omega$ must satisfy $\lim_{t \to \infty} \omega(t)=\infty$ is not needed in the proof of Proposition \ref{if u and -u are strongly paraconvex then u is unoformly differentiable}. Finally, observe that, since $\widetilde{\omega} \leq \omega,$ we have that $F\in C^{1,\omega}(X)$ as well. 
\end{proof}

\section{The class $C^{1, u}_{\mathcal{B}}(X)$}\label{sectionc1ub}

Let $X$ be a Hilbert space, and let $C^{1, u}_{\mathcal{B}}(X)$ stand for the space of all Fr\'echet differentiable functions on $X$ whose derivatives are uniformly continuous on each bounded subset of $X$.
 
In this section we combine Theorem \ref{MainThmBoundedCase} with a standard partition of unity in order to characterize the $1$-jets $(f, G)$ which are restrictions to $E$ of some $(F, \nabla F)$ with $F\in C^{1, u}_{\mathcal{B}}(X) $.

\begin{theorem}\label{thm for the class C1ub}
Let $(f, G)$ be a $1$-jet defined on a subset $E$ of a Hilbert space $X$, and suppose that $(f, G)$ is bounded on each bounded subset of $E$.\footnote{Note that this is a necessary condition for $(f,G)$ to be the restriction to $E$ of some $(F, \nabla F)$ with $C^{1, u}_{\mathcal{B}}(X) $.} Then there exists $F\in C^{1, u}_{\mathcal{B}}(X) $ with $(F, \nabla F)_{|_E}=(f, G)$ if and only if, for all bounded sequences $(x_n)\subset X$, $(y_n), (z_n)\subset E$ with  $|x_n-y_n|+|x_n-z_n|>0$,
$$
\lim_{n\to\infty} \left( |x_n-y_n|+|x_n-z_n| \right)=0 \implies \lim_{n\to\infty}\frac{f(y_n)+\langle G(y_n), x_n-y_n\rangle-f(z_n)-\langle G(z_n), x_n-z_n\rangle}{|x_n-y_n|+|x_n-z_n|}=0.
$$
\end{theorem}
\begin{proof}
For every $x\in X$, $y, z\in E$ with $|x-y|+|x-z|>0$, let us denote
$$
\theta(x,y,z) :=\frac{|f(y)+\langle G(y), x-y\rangle-f(z)-\langle G(z), x-z\rangle|}{|x-y|+|x-z|}.
$$
Thus the theorem states that $(f, G)=(F, \nabla F)$ for some $F\in C^{1, u}_{\mathcal{B}}(X) $ if and only if for all bounded sequences $(x_n)\subset X$, $(y_n), (z_n)\subset E$ with  $|x_n-y_n|+|x_n-z_n|>0$,
\begin{equation}\label{condition for C1ub extension}
\lim_{n\to\infty} \left( |x_n-y_n|+|x_n-z_n| \right)=0 \implies \lim_{n\to\infty}\theta(x_n, y_n, z_n)=0.
\end{equation}
Let us first show the necessity of this condition. Assume that $(F, \nabla F)=(f,G)$ on $E$ for some $F\in C^{1, u}_{\mathcal{B}}(X) $. For every modulus $\omega$ and every bounded set $B$ in $X$, let us denote
$$
A(F, \nabla F, \omega, B)=\sup_{x, y, z\in B, \, |x-y|+|x-z|> 0}\frac{ |F(y)+\langle \nabla F(y), x-y\rangle- F(z)-\langle \nabla F(z), x-z\rangle | }{ \varphi_{\omega}\left(|x-y|\right) +\varphi_{\omega}\left(|x-z|\right)},
$$
where $\varphi_{\omega}(t)=\int_{0}^{t}\omega(s)ds$. 
Then, for every ball $B$ in $X$, we have that
$$
M_{B, \omega}:=A(F, \nabla F, \omega, B)<\infty
$$
for some modulus $\omega$ depending on $\nabla F$ and $B$. Therefore, for all $x\in B,\, y, z\in E \cap B$ with $|x-y|+|x-z|>0$, using the fact that $\varphi_{\omega}$ is convex, we have
$$
\theta(x,y,z)\leq M_{B, \omega}\frac{\varphi_\omega(|x-y|)+\varphi_{\omega}(|x-y|)}{|x-y|+|x-z|} \leq M_{B, \omega} \frac{\varphi_{\omega}(|x-y|+|x-y|)}{|x-y|+|x-z|}\leq M_{B, \omega}\omega\left(|x-y|+|x-z|\right),
$$
which implies  \eqref{condition for C1ub extension}.

Conversely, if $(f, G)$ satisfies \eqref{condition for C1ub extension}, let us construct an $F\in C^{1, u}_{\mathcal{B}}(X) $ such that $(F, \nabla F)_{|_E}=(f, G)$.
Fix a function $\beta :\R\to [0,1]$ such that: $\beta$ is of class $C^{1,1}$; $\beta(t)>0 \iff t\in (-1, 1)$; $\beta(0)=1$; $\beta'(t)> 0 \iff t\in (-1, 0)$; $\beta(t)+\beta(t-1)=1$ for all $t\in [0,1]$, and $\textrm{Lip}(\beta)\leq 2$, and define
$$
\beta_{1}(t)=
\begin{cases}
1 & \textrm{ if } t\leq 1 \\
\beta(t-1) &  \textrm{ if }t\geq 1,
\end{cases}
$$
and for every $k\geq 2$,
$$
\beta_k(t)=\beta(t-k).
$$
Next, for every $k\in \N$, let us define $\psi_{k}:X\to [0,1]$ by
$$
\psi_k(x)=\beta_k\left( |x| \right),
$$
and denote $B_k=B(0, k)$ for each $k\in\N$, with $B_0=\emptyset$.
Then we have:
\begin{enumerate}
\item $\psi_k\in C^{1,1}(X)$ for all $k$;
\item $\sum_{k=1}^{\infty}\psi_k(x)=1$ for all $x\in X$;
\item $\textrm{supp}(\psi_{k})=B_{k+1}\setminus \textrm{int}(B_{k-1})$ for all $k\geq 2$, and $\textrm{supp}(\psi_1)=B_2$;
\item $\textrm{Lip}(\psi_k)\leq 2$ for all $k$.
\end{enumerate}
Also denote
$$
E_1 :=E\cap B_2, \,\,\, E_k :=E\cap \left( B_{k+1}\setminus \textrm{int}(B_{k-1})\right),
$$
and
$$
f_{k} :=f_{|_{E_k}}, \,\,\, G_k :=G_{|_{E_k}}.
$$
For each $k\in\N$ let us now define $\alpha_{k}:(0, \infty)\to [0, \infty)$ by
$$
\alpha_k(t) :=\sup\left\{ \theta(x, y, z) : x\in B_{3k}, \, y, z\in E\cap B_{3k}, 0<|x-y|+|x-z| \leq t \right\}.
$$
\begin{lemma}
{\em We have that $\lim_{t\to 0^{+}}\alpha_k(t)=0$.}
\end{lemma}
\begin{proof}
Otherwise there exist $\varepsilon>0$ and sequences $t_n\searrow 0$, $(y_n), (z_n)\subset E\cap B_{3k}$, $(x_n)\subset B_{3k}$ with $|x_n-y_n|+|x_n-z_n| \leq t_n$ and 
$$
\theta(x_n, y_n, z_n)\geq \varepsilon
$$
for all $n\in\N$, which contradicts \eqref{condition for C1ub extension}.
\end{proof}

Now let us set $\alpha_k(0)=0$. If $\alpha_k:[0, \infty)\to [0, \infty)$ is constantly $0$ then $G_k$ is constant, and for any $y_k\in B_k\cap E$ the function $F_k(x)=f(y_k)+\langle G(y_k), x-y_k\rangle$ has the property that $(F_k, \nabla F_k)=(f_k, G_k)$ on $E_k$. If $\alpha_k$ is not constant, we define $\gamma_k:[0, \infty)\to [0, \infty)$ by
$$
\gamma_k(t)=\inf\{ g(t) \, | \, g: [0, \infty)\to\R \textrm{ is concave and } g\geq \alpha_k\}
$$
(the concave envelope of $\alpha_k$). Then $\gamma_k$ is a nondecreasing continuous concave modulus of continuity. Define then $\omega_k:[0, \infty)\to [0, \infty)$ by
$$
\omega_k(t)=
\begin{cases}
\gamma_k(t) & \textrm{ if } 0\leq t\leq t_k \\
\gamma_k(t_k) +\gamma_{k}'(t_k)(t-t_k)  & \textrm { if } t\geq t_k,
\end{cases}
$$
where $t_k>0$ is some point at which $\gamma_k$ is differentiable and $\gamma_{k}'(t_k)>0$. Then $\omega_k$ is an increasing continuous concave modulus of continuity satisfying $\lim_{t\to\infty}\omega_k(t)=\infty$, $\gamma_k\leq \omega_k$ on $[0, \infty)$, and $\gamma_k=\omega_k$ on $[0, t_k]$. Next let us set
$$
\varphi_k(t)=\int_{0}^{t}\omega_k(s)ds, \,\,\, t\in [0, \infty).
$$
\begin{lemma}
{\em For each $k\in \N$ we have that 
$$
A(f_k, G_k, \omega_k) :=\sup_{x\in X; \, y, z\in E_k,\, |x-y|+|x-z|> 0}\frac{ |f(y)+\langle G(y), x-y\rangle- f(z)-\langle G(z), x-z\rangle | }{ \varphi_{k}\left(|x-y|\right) +\varphi_{k}\left(|x-z|\right)}<\infty.
$$}
\end{lemma}
\begin{proof}
From the above construction of $\omega_k$ and $\alpha_k$ it is clear that
$$
\theta(x, y, z)\leq \omega_k(|x-z|+|x-y|)\leq \omega_k(|x-z|)+\omega_k(|x-y|)
$$
for all $x\in B_{3k}$ and all $y,z\in E\cap B_{3k}$. This implies that
\begin{align*}
 |f(y)& +\langle G(y), x-y\rangle -f(z)-\langle G(z), x-z\rangle| \leq 
\left( |x-z|+|x-y|\right) \left( \omega_k (|x-z|)+\omega_k(|x-y|)\right) \\
 & \leq 2\max\{ |x-z|, |x-y|\} \, 2\max\{ \omega_k (|x-z|), \omega_k(|x-y|)\} \\
   & =4\max\{ |x-z|\omega_k (|x-z|),  |x-y|\omega_k(|x-y|)\} \leq 8\left(\varphi_k(|x-z|)+\varphi_k(|x-y|)\right)
\end{align*}
for all $x\in B_{3k}$ and all $y,z\in E\cap B_{3k}$. On the other hand, if $x\notin B_{3k}$ and $y, z\in E_k$ then we have that $|x-y|\geq 1$, $|x-z|\geq 1$, and using the convexity of $\varphi_k$ we get
\begin{align*}
 & \frac{|f(y)+\langle G(y), x-y\rangle -f(z)-\langle G(z), x-z\rangle|}{\varphi_{k}(|x-y|)+\varphi_{k}(|x-z|)}\leq \frac{2\|f_{k}\|_{\infty} +\|G_{k}\|_{\infty}\left(|x-y|+|x-z|\right)}{\varphi_{k}(|x-y|)+ \varphi_k(|x-z|)} \\
 & \qquad \leq \frac{ \left( \|f_{k}\|_{\infty} +\|G_{k}\|_{\infty} \right)\left(|x-y|+|x-z|\right)}{\varphi_{k}(|x-y|)+ \varphi_k(|x-z|)} \leq \frac{\|f_{k}\|_{\infty} +\|G_{k}\|_{\infty}}{\varphi_{k}(1)}.
\end{align*}
Therefore we have
$$
A(f_k, G_k, \omega_k)\leq\max\left\{ 8, \, \varphi_k(1)^{-1} \left( \|f_{k}\|_{\infty} +\|G_{k}\|_{\infty} \right) \right\} <\infty.
$$
\end{proof}

Now we can apply Theorem \ref{MainThmBoundedCase} to find a function $F_k\in C^{1, \omega_k}(X)$ such that $(F_k, \nabla F_k)_{|_{E_k}}=(f_k, G_k)$, with $(F_k, \nabla F_k)$ bounded. Let us finally define
$$
F=\sum_{k=1}^{\infty}\psi_k F_k.
$$
Since the sum defining $F$ is finite on every bounded subset of $X$, the functions $\psi_k$, $\nabla\psi_k$, $F_k$ and $\nabla F_k$ are bounded, and $\nabla\psi_k$ and $\nabla F_k$ are uniformly continuous on $X$, it is clear that $F\in C^{1, u}_{\mathcal{B}}(X) $. Also, using the facts that $\sum_{k=1}^{\infty}\nabla\psi_k=0$ and $F_k(y)=f_k(y)=f(y)$ and $\nabla F_k(y)=G_k(y)=G(y)$ if $y\in\textrm{supp}(\psi_k)\cap E$, we have that, for each $y\in E$, $F(y)=f(y)$ and
\begin{eqnarray*}
\nabla F(y)=\sum_{k=1}^{\infty}\psi_k(y)\nabla F_k(y)+ \sum_{k=1}^{\infty}F_k(y)\nabla\psi_k(y)
 =\sum_{k=1}^{\infty}\psi_k(y) G(y)+ F(y) \sum_{k=1}^{\infty} \nabla\psi_k(y)=G(y),
\end{eqnarray*}
hence $(F, \nabla F)_{|_E}=(f, G)$.
\end{proof}

\end{document}